\documentclass[11pt]{article}

\usepackage{amsfonts}
\usepackage{amssymb,amsmath,amsthm, enumerate}
\usepackage{latexsym}
\usepackage{fullpage}
\usepackage{hyperref}

\usepackage{graphicx}

\usepackage{color}

\newtheorem{theorem}{Theorem}[section]

\newtheorem{prop}[theorem]{Proposition}

\newtheorem{lemma}[theorem]{Lemma}

\newtheorem{defn}[theorem]{Definition}

\theoremstyle{definition}
\newtheorem{remark}[theorem]{Remark}

\newcounter{tenumerate}

\def\P{\mathbb{P}}

\renewcommand{\epsilon}{\varepsilon}
\newcommand{\e}{\varepsilon}

\DeclareMathOperator{\var}{Var} \DeclareMathOperator{\Cov}{Cov}

\newcommand{\E}{{\mathbb E}}
\newcommand{\remove}[1]{}
\renewcommand{\le}{\leqslant}
\renewcommand{\ge}{\geqslant}
\renewcommand{\leq}{\leqslant}
\renewcommand{\geq}{\geqslant}

\newcommand{\C}{\mathcal{C}}

\def\XXint#1#2#3{{\setbox0=\hbox{$#1{#2#3}{\int}$}
\vcenter{\hbox{$#2#3$}}\kern-.5\wd0}}

\def \b {\zeta}
\def \B {{\cal B}}
\def \BD {{\cal BD}}
\def \Cov {{\rm Cov}}

\def \e {\varepsilon}
\def \d {\delta}
\def \E {{\mathbb{E}}}
\def \k {\kappa}
\def \LR {{\cal LR}}
\def \P {{\mathbb{P}}}
\def \Var {{\rm Var}}
\def \s {\sigma}
\def \UD {{\cal UD}}
\def \( {\left( }
\def\) {\right) }
\def\[ {\left[}
\def\]{\right]}

\begin{document}

\title{{\bf Non-universality for first passage percolation on the exponential of log-correlated Gaussian fields}}

\author{ Jian Ding\thanks{Partially supported by NSF grant DMS-1313596, and NSF of China 11628101.} \\ University of Chicago \and Fuxi Zhang\thanks{Supported by NSF of China 11371040.}  \\
Peking University
}
\date{}

\maketitle

\begin{abstract}

We consider first passage percolation (FPP) where the vertex weight is given by the exponential of two-dimensional log-correlated Gaussian fields. Our work is motivated by understanding the discrete analog for the random metric associated with \emph{Liouville quantum gravity} (LQG), which roughly corresponds to the exponential of a two-dimensional Gaussian free field (GFF).

The particular focus of the present paper is an aspect of universality for such FPP among the family of log-correlated Gaussian fields. More precisely, we construct a family of log-correlated Gaussian fields, and show that the FPP distance between two typically sampled vertices (according to the LQG measure) is $N^{1+ O(\epsilon)}$, where $N$ is the side length of the box and $\epsilon$ can be made arbitrarily small if we tune a certain parameter in our construction. That is, the exponents can be arbitrarily close to $1$. Combined with a recent work of the first author and Goswami on an upper bound for this exponent when the underlying field is a GFF, our result implies that such exponent is \emph{not} universal among the family of log-correlated Gaussian fields.
\end{abstract}

\section{Introduction}
For an $N\times N$ box $V_N\subseteq \mathbb Z^2$ with left bottom corner at the origin, we consider a log-correlated Gaussian field $\{\varphi_{N,v}: v\in V_N\}$ (see below for precise definitions). The main object investigated in the present article is the first passage percolation (FPP) on the exponential of  $\{ \varphi_{N,v}: v\in V_N \}$. More precisely, for $\gamma>0$ and $u, v\in V_N$, we define the FPP distance by
\begin{equation}\label{eq-def-LQG}
d_\gamma(u, v) = \min_{P} \sum_{w\in P} e^{\gamma \varphi_{N, w}}\,,
\end{equation}
where the minimum is taken over all paths in $V_N$ connecting $u$ and $v$.

Our motivation is from the two-dimensional Liouville quantum gravity (LQG) introduced in \cite{P81}. The mathematical model for LQG  can be formally described as a Riemannian ``manifold" with tensor of the form
\begin{equation}\label{eq-LQG-tensor}
e^{\gamma X(x)} dx^2 ,
\end{equation}
where $X$ is a Gaussian free field (GFF) on (say) a torus $\mathbb T$ and $\gamma \in [0, 2)$ is a parameter. Note the realization of the GFF is a generalized function rather than a function. A fundamental question yet to be understood is on the LQG geometry. The volume of this ``manifold" is well-defined, say set up by the theory of Gaussian multiplicative chaos for log-correlated Gaussian fields \cite{Kahane} (see also \cite{RV14}), and is (in some context) referred to as the LQG measure. In a celebrated work \cite{DS11}, it is proved that the famous KPZ formula \cite{KPZ88} (for more background and references, see \cite{DS11}) holds for the LQG measure.  However, people know little about the metric of this ``manifold". Despite intensive research, it remains an important open problem to even rigorously make sense of \eqref{eq-LQG-tensor} except when $\gamma = \sqrt{8/3}$. For seminal works in this direction, see \cite{MS15, BMS15, MS15b}.

In light of \eqref{eq-LQG-tensor}, it is natural to consider the discrete analogue of the LQG metric on a two-dimensional discrete Gaussian free field $\{\eta_{N, v}\}$ (say, with Dirichlet boundary condition by convention), which is a special instance of log-correlated Gaussian fields with covariance $\E \eta_{N, u} \eta_{N, v}$ given by the Green function of simple random walk killed at exiting the boundary $\partial V_N$. For $u, v\in V_N$, we define the discrete analog of the LQG measure as a random probability measure $\mu_\gamma$ on $V_N$ by
\begin{equation}\label{eq-measure}
\mu_\gamma(v) \propto e^{\gamma \eta_{N, v}}
\end{equation}
(we remark that the LQG measure has also been studied from the spin glass point of view, see \cite{AZ13, AZ14}). Furthermore, by resemblance in the formulation between \eqref{eq-def-LQG} and \eqref{eq-LQG-tensor}, one way to define the discrete analog for the LQG metric is to consider the FPP distance, i.e., replacing the field $\{\varphi_{N, v}\}$ in \eqref{eq-def-LQG} with $\{\eta_{N, v}\}$ (such metric was explicitly mentioned in \cite{Benjamini}). It is interesting to understand the FPP metric, part of it due to its natural and close connection to the continuous LQG metric. A major open problem is to prove that a (suitable) discrete LQG metric under appropriate normalization converges in some sense to the continuous LQG metric. While the FPP metric might not eventually be ``\emph{the suitable}'' one with all properties that are desirable for the LQG metric, we believe it should capture a large part of the fundamental mathematical structure of the suitable notion.

We choose to work with the FPP metric for its simple formulation, as well as its connection  to classical first passage percolation with independent weights (see, e.g., \cite{GK12, ADH15} for recent surveys). Another  motivation of studying the FPP metric is its connection to the heat kernel estimate for Liouville Brownian motion (LBM), whose mathematical construction is provided in \cite{GRV14, B14}. The LBM is closely related to the geometry of LQG; in \cite{FM09, BGRV14} the KPZ formula is derived from the heat kernel of LBM. In \cite{MRVZ14} some nontrivial bounds for the heat kernel of LBM are established. A very interesting direction is to compute the heat kernel of LBM with high precision. It is plausible that understanding the FPP metric as well as its geodesics is of crucial importance in computing the heat kernel of LBM. In fact, in a recent work \cite{DZZ17}, the exponent for the Liouville heat kernel on log-correlated fields constructed in \eqref{eq-def-field} was computed, building on results in the present article.

In this paper, we study an aspect of universality for the FPP metric when the underlying random media is in the class of log-correlated Gaussian fields. Our motivation is two-fold. On the one hand, universality is interesting on its own. There has been an intensive research recently on maxima of log-correlated Gaussian fields, and many universality aspects have been observed and rigorously established. It is shown that the limiting law of the centered maximum is in the same universality class for all log-correlated Gaussian fields under mild assumptions  (see \cite{DRZ15, Madaule} and references therein). Furthermore, it is highly expected that the limiting law of the extremal process should also be in the same universality class (see \cite{BL13, BL14} for the case of discrete Gaussian free field and see \cite{ABK13, ABBS} for the case of branching Brownian motion). From the spin glass point of view, the works \cite{AZ13, AZ14} strongly suggest certain universality behavior of the Gibbs measure (corresponding to the LQG measure in our context) among log-correlated Gaussian fields. Furthermore, the work of \cite{RV11} suggests that the KPZ formula for the LQG measure is universal among log-correlated Gaussian fields (in the continuous set-up). Since the exponent of the FPP metric between a typical pair is largely determined by the geometry of the random set whose Gaussian values are in the same order of the maximum, it \emph{seems} plausible that this may also have a universal exponent among the family of log-correlated Gaussian fields. On the other hand, the universality, either holding or not, is a useful guide for  computing the exponent of the FPP metric for the GFF since it will suggest what are the ``correct'' properties to focus on in order for such computations.

However, despite the fact that universality is of much interest, we are not aware of a precise formulation in this particular case. In what follows, we propose a natural candidate, for which we refer to as the \emph{strong universality} of the FPP metric.

Following \cite{DRZ15}, we say a suitably normalized version of Gaussian field $\{\varphi_{N, v}: v\in V_N\}$ is log-correlated if it satisfies the following two properties (A.0) and (A.1). Write $w = (w_1, w_2)$ for all $w \in \mathbb{Z}^2$, and denote  $|u-v| = \max \{ |u_1-v_1|, |u_2 - v_2| \}$ the $\ell_\infty$-distance.
\begin{itemize}
	\item[(A.0)] \textbf{(Logarithmically  bounded fields)} \
There exists a constant $\alpha_0 >0$ such that for all $u, v\in V_N$, \\
 $$\var (\varphi_{N, v} ) \leq \log_2 N + \alpha_0$$ and
$$\E (\varphi_{N, v} - \varphi_{N, u})^2  \leq 2 \log_2 ( |u-v| \vee 1) - | \var (\varphi_{N, v}) - \var (\varphi_{N, u} ) | + 4\alpha_0 . $$
\end{itemize}

\begin{itemize}
	\item[(A.1)] \textbf{(Logarithmically correlated fields)} \
For any $\delta>0$ there
exists a constant $\alpha^{(\delta)}>0$ such that for all $u, v\in V_N^\delta$,
		$|\Cov( \varphi_{N, v},\varphi_{N, u}) - (\log_2 N - \log_2 (|u-v| \vee 1)
		)| \leq \alpha^{(\delta)}$. Here  we denote  $V_N^\d = \{ u \in V_N : d (u, V_N^c) > \d \}$, where $d(u, V_N^c) = \min \{ |u-v| : v \in V_N^c \}$.
\end{itemize}

The seemingly odd assumption on (A.1) (with introduction of $V_N^{\delta}$) is to give enough flexibility to incorporate the influence from the boundary, seen in the GFF as well as the four-dimensional Gaussian membrane model. At this point, it is natural to define the measure $\mu_{\gamma}$ and the metric $d_{\gamma}$ with respect to the field $\{\varphi_{N,v}: v\in V_N\}$, in the same manner as in \eqref{eq-measure} (where we simply replace the field $\eta_{N, v}$ with $\varphi_{N, v}$) and \eqref{eq-def-LQG}.

\begin{defn}[Strong universality]
For each $\gamma >0$, there exists $\beta = \beta(\gamma)$ such that the following holds for any sequence of log-correlated Gaussian fields (and in particular, $\beta$ does not depend on $\alpha_0$ or $\alpha^{(\delta)}$ in (A.0) and (A.1)). For any fixed $\epsilon>0$, with probability tending to 1 as $N\to \infty$,
 $$
\mu_\gamma\times \mu_\gamma(\{(u, v): N^{\beta - \epsilon} \leq d_\gamma(u, v) \leq N^{\beta + \epsilon}) > 1 - \e_N,
 $$
for a sequence of numbers  $\e_N \to 0$ (as $N\to \infty$).

\end{defn}

In this paper, we construct a family of log-correlated Gaussian fields. The main result (Theorem~\ref{maintheorem}) is that for $0< \gamma <1/2$, we can make $\beta(\gamma)$ arbitrarily close to 1 if we tune a certain parameter in our construction. Previously, precise formulae on exponents for highly related random metrics associated with the GFF were predicted \cite{Watabiki93, ANRBW98, AB14}, which suggested that $\beta(\gamma) < 1$ for the case of the GFF. In addition, in a recent work \cite{DG15} it was proved that $\beta(\gamma)<1$ for small and fixed $\gamma>0$ when the underlying field is a two-dimensional branching random walk (which, approximately, is a log-correlated Gaussian field); furthermore, during the submission of the present article, it was established in \cite{DG16} that $\beta(\gamma)<1$ for the case of the GFF with small and fixed $\gamma$. Altogether, this implies that the strong universality does \emph{not} hold.

\smallskip

\noindent {\bf $K$-coarse modified branching random walk.} Our construction is based on the modified branching random walk (MBRW) introduced by \cite{BZ10}. We first briefly review the definition of the MBRW in $\mathbb Z^2$. Suppose $N=2^n$ for some $n \in \mathbb{N}$. For $j=0,1,\ldots, n$ we define $\mathcal{B}_j$ to be the set of boxes of side length $2^j$ with corners in $\mathbb{Z}^d$. For $v \in V_N$, we define $\mathcal{B}_j (v)$ to be those elements of $\mathcal{B}_j$ which contain $v$.
Denote by  $\{ b_{j,B} : j \ge 0, B \in \mathcal{B}_j \}$ a family of independent centered Gaussian variables such that $ \var ( b_{j,  B}) = 2^{-2j}$ for all $B\in \mathcal B_j$. Then the MBRW $\{ \mathcal{S}_{N, z} \}_{z \in V_N}$ is defined by
 $$
\mathcal{S}_{N, z} = \sum_{j=0}^n \sum_{B \in \mathcal{B}_j (z)} b_{j,B} \,.
 $$
Note that our definition of the MBRW is slightly different from that in \cite{BZ10} since we do not view the box as a torus, but this difference is only for our technical convenience and is immaterial.
For an integer $K = 2^k$, we define the $K$-coarse MBRW by
 \begin{equation} \label{eq-def-field}
\varphi_{N, z} = \varphi_{N, z}(K) =  \sum_{j=0}^{\lfloor n/k \rfloor - 1} \sum_{B \in \mathcal{B}_{jk} (z)} \sqrt{k} b_{jk,B} \ ,
 \end{equation}
where $\lfloor x \rfloor$ is the largest integer which is less than or equal to $x$ (and analogously, we will use the notation $\lceil x \rceil$ for the smallest integer which is greater than or equal to $x$). Following the proof of \cite[Lemma 2.2]{BZ10}, we will show that for $u,v \in V_N$,
 $$
|\Cov(\varphi_{N, u}, \varphi_{N, v})-(n - \log_2 ( |u-v| \vee 1))| \le 6 k + 1 \,,
 $$
(see Lemma \ref{covariance} below). Therefore, a sequence of $K$-coarse MBRW (for fixed $K$ with $N\to \infty$) is a sequence of log-correlated Gaussian fields (where all the $\alpha_0$ and $\alpha^{(\delta)}$ are bounded by $6k+1$).

 \begin{theorem} \label{maintheorem}
Fix $0 < \gamma < \frac 1 2$. For any $\epsilon>0$ there exist $K(\epsilon)$ and $\rho = \rho (\e) > 0$ such that for all $K \ge K (\e)$, the FPP metric for the $K$-coarse MBRW satisfies
 $$
\lim_{N\to \infty} \P \( \mu_\gamma\times \mu_\gamma \( \big\{ (u, v): N^{1 - \epsilon} \leq d_{\gamma}(u, v)  \leq  N^{1+\epsilon} \big\} \) \geq 1 - N^{-\rho} \) = 1 .
 $$
 \end{theorem}
 \begin{remark}
The assumption $\gamma<1/2$ is in some sense necessary in order to control the influence to the field in the local neighborhood of $v$ and $u$ when conditioning on the values of $\varphi_{N, v}$ and $\varphi_{N, u}$. The larger the $\gamma$ is, the larger the typical values of $\varphi_{N, v}$ and $\varphi_{N, u}$ are when $v$ and $u$ are sampled according to the measure $\mu_\gamma$. It is clear that when $\gamma$ exceeds a certain threshold, the statement in the preceding theorem no longer holds. However, we do not attempt to achieve the sharp threshold in this work.
\end{remark}
\begin{remark}
It remains to see what the non-universality result in our work would imply in the continuous set-up. An interesting future direction is to investigate the behavior of the geodesics for the FPP metric associated with the $K$-coarse MBRW, and to decide whether their dimensions are strictly larger than 1 (which is predicted to be the case for the GFF and confirmed by \cite{DG16,DZ16} when $\gamma$ is small). From the current work, we know that there exists a path of cardinality at most $N^{1+ \epsilon}$ which has sum of weights less than $N^{1+\epsilon}$ (see Theorem~\ref{thmgoodpath}), where $\epsilon$ vanishes as $K\to \infty$.
\end{remark}

We next describe the main ingredient in the proof of Theorem \ref{maintheorem}.
For applications in \cite{DZZ17} on computing the heat kernel for the LBM on the continuous-analog of our $K$-coarse MBRW,  we introduce a Bernoulli process  $\{ \xi_{N, w}, w \in V_N \}$ (for the purpose of proving Theorem~\ref{maintheorem}, we simply set $\xi_{N,w} = 1$ for all $w$). Let $q$ be a positive integer, $0 < p \le 1$, and we  assume
\begin{itemize}
\item $\{ \xi_{N, w}, w \in V_N \}$ is independent of $\{ \varphi_{N,w}, w \in V_N \}$\,.

\item  $\{\xi_{N, w}\}$ is $q$-dependent, i.e.,  $\{  \xi_{N,w}, w \in A \}$ is independent of $\{  \xi_{N,w}, w \in B \}$ provided $d (A, B) > q$. Here $d (A, B) = \min \{ |z-w| : z \in A, w \in B \}$.

\item  $\P (\xi_{N,w} = 1 ) \ge p$ for all $w \in V_N$.
\end{itemize}
 Using the language of percolation, we say a site $w$ is open if $\xi_{N,w} = 1$, and a path is open if all sites on it are open. For convenience, we define the  {\em diameter} of a path to be the maximal $\ell_\infty$-distance of two vertices on the path.
 \begin{theorem} \label{cardinality}
Let $\delta \in (0,1)$, $\k  \in (0, \frac 1 2 \d^2)$. Then, there exist $K_0 = K_0 (\d, q)>0$ and $p_0 = p_0 (\d, q) \in (0,1)$ such that the following holds. Suppose $K \ge K_0$, and $p \ge p_0$. Then, for $N$ large enough, with probability at least $1 - 6 e^{- n/10}$,
 $$
| \{ w \in P: \varphi_{N, w}(K) \leq 4 \delta \log N, \ \xi_{N,w} = 1 \} | \ge N^{1-\d}
 $$
simultaneously for all paths $P$ with diameter at least $N^{1-\k }$.
 \end{theorem}
 \begin{remark}
The geometric property established in Theorem~\ref{cardinality} is formulated without referring to any notion of the LQG metric.  We \emph{believe} this is the key property making the $K$-coarse MBRW differ from the GFF and leading to the non-universality of the FPP metric. Furthermore, we \emph{feel} that due to such difference the non-universality \emph{should} hold for any reasonable notion of the discrete LQG metric.
  \end{remark}

By duality, Theorem~\ref{cardinality} implies that for a rectangle with height at least $N^{1-\k}$, there exist $N^{1-\d}$ horizontal cut sets with $w$ being open and $\varphi_{N, w}(K) $ being small. This implies the existence of a ``short" open horizontal path with small Gaussian values on it. The following result can be deduced from Theorem~\ref{cardinality}, combined with a Russo-Seymour-Welsh kind of constructions \cite{R,SW}.
 \begin{theorem} \label{thmgoodpath}
Let $\d \in (0, \frac 1 2)$,  $\k  \in (0, \frac 1 2 \d^2)$, $\b \in (0, \frac 1 4 \d^4 \log^2 2 )$. Then, there exist $K_1 = K_1 (\d, q)>0$ and $p_1 = p_1 (\d, q) \in (0,1)$ such that the following holds. Suppose $K \ge K_1$, and $p \ge p_1$. For $N$ large enough, if $|u-v| > N^{1-\k}$, then with probability at least $1 - 128 (1-p)^{\frac 1 {(2q+1)^2}} - e^{-\b n}$, one can find an open path $P$ connecting $u$, $v$ such that $|P| \le N^{1+2 \d}$, and $\varphi_{N,w} (K) \le 7 \d \log N$, $\forall w \in P$.
 \end{theorem}

\noindent {\bf Organization.} Theorem~\ref{cardinality}, the core theorem of the article, is proved in Section \ref{pfofcard}. In Section~\ref{pfofthm}, we show that (see Proposition~\ref{shortdist}) a typical pair of vertices $u,v$ sampled from $\mu_\gamma \times \mu_\gamma$ has $\ell_\infty$-distance at least $N^{1-\k}$. Combined with Theorem~\ref{cardinality}, it yields the desired lower bound in Theorem~\ref{maintheorem}. In Proposition~\ref{goodpath}, we show the existence of a certain \emph{good} path connecting two vertices, which implies Theorem~\ref{thmgoodpath}. This gives the desired upper bound in Theorem~\ref{maintheorem}, modulo the difference between two randomly sampled vertices according to the LQG measure and two fixed vertices (which is addressed in Proposition~\ref{givenendvalues}).

\noindent{\bf Notation convention}. Throughout the paper, $\d$, $\k$, $\zeta$ are small positive parameters, $q$ is a fixed positive integer, $p \in (0,1]$ and $k$ is a large parameter chosen depending on $\delta$ and $q$. Denote $m = \lfloor n/k \rfloor$. We consider the limiting behavior when $n\to \infty$, thus we can assume without loss of generality that the inequalities such as $2 \d n - \log n + 2^k > \d n$ hold.

\noindent {\bf Acknowledgement.} We warmly thank Ofer Zeitouni, Pascal Maillard, Steve Lalley, Marek Biskup, R{\'e}mi Rhodes and Vincent Vargas for many helpful discussions.

\section{Connectivity for level sets} \label{pfofcard}

This section is devoted to the proof of Theorem~\ref{cardinality}. We will first give the heuristic argument in the simpler case where $p=1$.  The proof is by contradiction. Suppose there exists a path $P$ of diameter $\geq N^{1-\k}$ such that all the Gaussian values along the path are $\geq 4\delta \log N$ except for $N^{1-\delta}$ vertices (recall that $\k <\delta^2/2$). We consider the average of all the Gaussian variables along the path $P$, i.e., $\mathrm{Av}_P = \frac{1}{|P|} \sum_{z\in P} \varphi_{N, z}$. Then $\mathrm{Av}_P \ge 3 \d \log N$. For the convenience of notation, we write
\begin{equation}\label{eq-def-psi}
\psi_{j,z}  = \sum_{B \in \B_{jk} (z)} \sqrt{k} b_{jk, B} .
 \end{equation}
Recall $m = \lfloor n / k \rfloor$ and $\varphi_{N,z} = \sum_{j=0}^{m-1} \psi_{j,z}$. Then it follows that
$$\mathrm{Av}_{P, j^*} \ge 3 \d k\mbox{ for some }j^*\,,$$ where $\mathrm{Av}_{P, j^*}$ is the average in the $j^*$-th level, i.e., $\mathrm{Av}_{P, j^*} = \frac{1}{|P|} \sum_{z\in P} \psi_{j^*, z}$. We consider the sequence of boxes $B_1, \ldots, B_H$ (with $H \ge N^{1-\k} 2^{-j^* k}$) of side length $2^{j^* k}$ along $P$, and let
$$T_h = \max_{z\in B_h} \psi_{j^*, z}\,.$$
It is then natural (but incorrect as we will point out later) to expect that the averaged value of $T = \{  T_h \}_{h=1,\ldots,H}$ along the sequence of boxes $\{ B_1, \ldots, B_H \}$ is larger than $3 \d k$, i.e.,
 $$
\mathrm{Av}_{T, \{B_1, \ldots, B_H\}, j^*} : = \frac 1 H (T_1 + \ldots + T_H) \ge 3 \d k\,.
 $$
But this can be ruled out by the following  first moment computation. One can show that  $T_1, \ldots, T_H$ are essentially independent of each other and each $T_h$ has mean $O(\sqrt{k})$ with sub-Gaussian tail of variance $O(k)$. This implies that (when $k$ is sufficiently large) for each such sequence of boxes $B_1, \ldots, B_H$,
 $$\P( \mathrm{Av}_{T, \{B_1, \ldots, B_H\}, j^*} \geq 3\delta k) \leq 2 e^{-4 \delta^2 k H}\,.$$
However, the number of such sequences of boxes is at most $4^H$. Therefore, a simple union bound implies that with overwhelming probability, there exists no sequence of boxes with $\mathrm{Av}_{T, \{B_1, \ldots, B_H\}, j^*} \geq 3\delta k$, arriving at a contradiction.

A crucial gap in this heuristic is that the path $P$ could intersect certain $B_h$ for many times (up to $2^{2 j^* k}$) and intersect another $B_{h^\prime}$ for $O(2^{j^*k})$ times, thus  $\mathrm{Av}_{P,j^*}$ is dominated by a \emph{weighted} average of $T_1, \ldots, T_H$ rather than $\mathrm{Av}_{T, \{B_1, \ldots, B_H\}, j^*}$. In order to address this issue, in Section \ref{S.blocknest} we extract a subset $Q$ of a path $P$, which stretches uniformly in some sense (as explained at the beginning of Section~\ref{S.blocknest}). Thus, the average of $\psi_{j^*, z}$ along $Q$ is indeed dominated by the average $\mathrm{Av}_{T, \{B_1, \ldots, B_H\}, j^*}$, where $B_1, \ldots, B_H$ is a sequence of boxes associated with $Q$.  Afterwards, the proof of Theorem \ref{cardinality} is  carried out in Section \ref{S.pfofcard}, where we will apply this heuristic to $Q$ (rather than $P$).

\subsection{Preliminary lemmas} \label{S.lemmas}

This section records a number of preliminary lemmas.

\begin{lemma} \label{covariance}
Denote $\s_{z,w} = \E \varphi_{N,z} \varphi_{N,w}$. Then, $0 \le (km - \log_2 (|z-w| \vee 1))  - \s_{z,w} \le  5k + 1$.
 \end{lemma}
\begin{proof}
The proof follows that of Lemma 2.2 in \cite{BZ10}. The case for $z = w$ is obvious and in what follows we will focus on the case when $z \neq w$. Note $\E \psi_{j,z}  \psi_{j,w} = 0$ if $|z-w| \ge 2^{jk}$. Suppose $|z-w| < 2^{jk}$, equivalently, $j \ge \lceil \frac 1 k \log_2 (|z-w|+1) \rceil $. Then
 $$
\E \psi_{j,z}  \psi_{j,w}
 =
k  (1 - \frac { |z_1 - w_1| } {2^{jk}}) (1 - \frac { |z_2 - w_2| } {2^{jk}}),
 $$
which implies that
 \begin{equation} \label{covofpsi}
k - 2k \times \frac {|z-w|}{2^{jk}} \le \E \psi_{j,z}  \psi_{j,w}  \le k.
 \end{equation}
Note $\s_{z,w}  = \sum_{j=\lceil \frac 1 k \log_2 (|z-w|+1) \rceil}^{m-1} \E \psi_{j,z}  \psi_{j,w} $. On the one hand,
 $$
\s_{z,w} \le k \( m - \lceil \frac 1 k \log_2 (|z-w|+1) \rceil \) \le km - \log_2 |z-w|  .
 $$
On the other hand,
 \begin{align*}
\s_{z,w}
  & \ge
\sum_{j=\lceil \frac 1 k \log_2 (|z-w|+1) \rceil}^{m-1} k (1 - 2 \times \frac {|z-w|}{2^{jk}})
 \\ & =
k \( m - \lceil \frac 1 k \log_2 (|z-w|+1) \rceil \) - 4k |z-w| \times 2^{- k  \lceil \frac 1 k \log_2 (|z-w|+1) \rceil }
 \\ & \ge
km  - \( \log_2 |z-w| + k + 1 \) - 4k
 \\ & =
km - \log_2 |z-w| - 5 k - 1,
 \end{align*}
 completing the verification of the lemma.
 \end{proof}

 \begin{lemma} \label{concentration}
(\cite[Theorem 7.1, Equation (7.4)]{L01}) Let $\{ G_z : z \in B \}$ be a Gaussian field on a finite index set $B$. Set $\s^2 = \max_{z \in B} \Var (G_z)$. Then
 $$
\P (| \max_{z \in B} G_z - \E \max_{z \in B} G_z | \ge x) \le 2 e^{-\frac {x^2}{2 \s^2}} .
 $$
 \end{lemma}
 \begin{lemma}\label{d2mtom}
We have that
 $$
\P \( \mbox{$\sum_{j=\lceil (1 - \d^2) m\rceil}^{m-1}$}   \psi_{j,z} \le \frac 5 2 \d \log N, \ \ \mbox{ for all } \ z \in V_N \) \ge 1 - 2 e^{-  n/10 } .
 $$
 In addition, we have
 \begin{equation}\label{eq-continuity}
 \P(|\varphi_{N, u} - \varphi_{N, v}| \leq 100 \sqrt{K\log N}\,, \mbox{ for all } u, v\in V_N, u\sim v) \geq 1 - e^{-n/10 } .
 \end{equation}
 \end{lemma}
\begin{proof}
The proof of these two equalities are highly similar to each other, and thus we omit the proof of \eqref{eq-continuity}.
Denote $\theta_z =  \sum_{j=\lceil (1 - \d^2) m\rceil }^{m-1} \psi_{j,z}$ for short. Note for each $z$, $\Var (\theta_z) \le k\d^2 m \le \d^2 n$.  Thus,
 $$
\P \( \theta_z > \tfrac5 2 \d \log N \) \le 2 e^{- \frac 1 {2 \Var (\theta_z)} \frac {25}4 \d^2 \log^2 N} \le  2 e^{- \frac {25 \log^2 2} 8 n } .
 $$
 Therefore, a simple union bound  yields that
 \belowdisplayskip=-12pt
\begin{align*}
\P \( \exists z \in V_N {\rm \ such \ that \ } \theta_z > \tfrac 5 2 \d \log N \)
 \le
2 e^{- \frac {25 \log^2 2} 8 n}   N^2 \leq  2 e^{- n/10 }.
 \end{align*} \qedhere
\end{proof}

 \begin{lemma} \label{expmean}
Suppose $B \subset V_N$. Then for all $a > 1 / \sqrt{2 \pi k}$,
 $$
\E \exp \(  a  ( \max_{z \in B} \psi_{j,z} - \E \max_{z \in B} \psi_{j,z} ) \) \le 3 \sqrt{2 \pi ka^2} e^{\frac 1 2 ka^2 } .
  $$
 \end{lemma}
\begin{proof} Denote $T = \max_{z \in B} \psi_{j,z}$ for short. Then,
 \begin{eqnarray*}
\E e^{ a ( T - \E T )}
 & \le &
1 + \int_1^\infty \P \( e^{a(T - \E T )} > x \) d x
 \\ & \le &
1 + \int_0^\infty \P \(  T - \E T > \frac x a \) e^x d x.
 \end{eqnarray*}
By Lemma \ref{concentration}, $\P \(  T - \E T > \frac x a \) \le 2 \exp (-\frac {(x/a)^2}{2 k})$. It follows that

 \begin{eqnarray*}
\E e^{ a ( T - \E T )}
 & \le &
1 + \int_0^\infty 2 e^{- \frac {x^2}{2ka^2}} e^x d x
 \\ & = &
1 + \int_0^\infty 2 e^{- \frac 1{2ka^2} (x - ka^2)^2 + \frac 1 2 ka^2} d x
 \\ & \le &
1 + 2 e^{\frac 1 2 ka^2 } \sqrt{2 \pi ka^2
} \le
3 e^{\frac 1 2 ka^2 } \sqrt{2 \pi ka^2},
 \end{eqnarray*}
where in the last inequality we use $1 \le \sqrt{2 \pi ka^2}$.
\end{proof}

 \begin{lemma} \label{meanofmax}
(\cite[Theorem 1.2]{DRZ15}) Let $m_N = 2 \sqrt {\log 2} n - \frac {3}{4 \sqrt{\log 2}} \log n$. Then there exists a constant $C_K$, depending on $K$, such that
 $$
| \E \max_{z \in V_N} \varphi_{N,z} - m_N | \le C_K .
 $$
 \end{lemma}

 \begin{lemma}\label{maxinbox}
(\cite[Theorem 4.1]{A90}) There exists a universal constant $C_F$ with the following property. Let $B \subset \mathbb{Z}^2$ be a box of side length $b$ and $\{ G_w : w \in B \}$ be a mean zero Gaussian field satisfying
 $$
\E (G_z - G_w)^2 \le |z-w|/b \ \ \ \mbox{for all \ } z,w \in B.
 $$
Then $\E \max_{w \in B} G_w \le C_F$.
 \end{lemma}

 \begin{lemma} \label{Emax}
For all $j$ and $B$ of side length $2^{jk} b$ (with $b$ being a positive integer), we have that $\E \max_{z \in B} \psi_{j,z}  \le 2 C_F \sqrt {kb}$.
 \end{lemma}
\begin{proof} Suppose $z,w \in B$. If $| z - w| < 2^{jk}$, by \eqref{covofpsi},
 $$
\E (\psi_{j,z} - \psi_{j,w} )^2 = 2 k - 2 \E \psi_{j,z}  \psi_{j,w}  \le 4 k \times \frac { |z-w| }{2^{jk}}.
 $$
Otherwise, $\E (\psi_{j,z} - \psi_{j,w} )^2 = 2k \le 4 k \times \frac { |z-w| }{2^{jk}}$.
By Lemma \ref{maxinbox}, $\E \max_{z \in B} \psi_{j,z}  \le C_F \sqrt {4kb}$.
\end{proof}

\subsection{The $k$-block-nest} \label{S.blocknest}

This section is devoted to extract a subset $Q$ of a path $P$, which ``stretches uniformly" (roughly speaking, for boxes of the same side length along $P$, the subset $Q$ has the same number of points in each box). Our construction is based on an inductive procedure, and we first give a description on the mental picture behind it. Let $\BD_{jk}$ be those elements in $\B_{jk}$ which have lower left corners in $2^{jk} \mathbb{Z}^2$. We see that $(\BD_{jk})_{j=m, \ldots, 0}$ forms an increasing sequence of nested uniform partitions of $V_N$. A natural attempt for the initial round of the induction is to consider the intersections of $P$ with all boxes in $\BD_{(m-1)k}$, which lead to a few subpaths of $P$. However, there are a couple of issues with this oversimplified attempt:
\begin{itemize}
\item Despite that the path $P$ is connected and self-avoiding, the intersection of $P$ with a box in $\BD_{(m-1)k}$ could contain a number of disconnected subpaths (the disconnectedness can potentially result in difficulties in entropy estimates).

\item Some of the boxes in  $\BD_{(m-1)k}$ may intersect $P$ substantially more than other boxes, which would violate our desired ``uniformly stretching'' property of $Q$.
\end{itemize}

Fix a box $B$ in $\BD_{(m-1)k}$. In order to address the issue of disconnectedness, we only take the ``first'' connected piece of $P$ in $B$, denoted by $P^{(1)}$. Furthermore, in order to address the issue of non-uniformity, we only keep an initial segment $P^{(2)}$ of $P^{(1)}$ so that $P^{(2)}$ intersects with precisely $K_1$ boxes in $\BD_{(m-2)k}$ for a pre-fixed number $K_1$ (depending on $K$). However, in the fix for the second issue, we face the difficulty that $P^{(1)}$ might be short so that it is impossible to extract such $P^{(2)}$. In order to address this issue and carry out the induction, we will define {\em traversals} of a box, and replace the concept of intersection with traversal. Then, from a traversal $P^{(1)}$, we are able to extract $P^{(2)}$ to traverse exactly $K_1$ boxes in $\BD_{(m-2)k}$. This ensures that our induction would in the end give us a large subset $Q\subseteq P$ which stretches uniformly.  In what follows, we will provide the formal construction and verification for the aforementioned inductive procedure, which we will refer to as the \emph{$k$-block-nest program}.

Suppose $B$ is a box of side length $\ell$. Let $B^*$ and $B^{**}$ be the boxes centered at $B$ of side lengths $3 \ell$ and $7 \ell $, respectively. That is,
 $$
B^* = \{ z : d (z, B) \le \ell \},  \ \ \ B^{**} = \{ z : d (z, B) \le 3 \ell \} ,
 $$
where $d(z,B) : = \min \{ |z-w| : w \in B \}$. By {\em traversal} of a box $B$, we mean a path in $B^*$ from $\partial B$ to $\partial B^*$. By \emph{distance} of a path, we mean the $\ell_\infty$-distance between the two end points of the path. Then each traversal of $B$ has distance at least the side length of $B$. First, let $j \ge 1$. We aim to extract traversals of $K_1$ (to be defined) boxes in $\BD_{jk}$ from a traversal of a box in $\BD_{(j+1)k}$.

 \begin{defn}
Suppose $B_1, B_2, \ldots, B_H \in \BD_{jk}$. They are called $j$-separated if $d( B^*_h, B^*_{h^\prime} ) > 2^{jk}$, $\forall h \neq h^\prime$, and coherent if in addition $d (B_{h+1}, B^{**}_h) = d (B_{h+1}, \cup_{s=1}^h B^{**}_s) = 1$, $\forall h = 1, \cdots, H-1$.
 \end{defn}

 \begin{remark}
The definition of $j$-separation is to ensure the independence of $\max_{z \in B^*_h} \psi_{j,z}$, $1 \le h \le H$. The further definition of coherence enables us to estimate the number of all such sequences of boxes. Both of these two aspects are essential in the heuristic argument at the beginning of Section~\ref{pfofcard}.
 \end{remark}

 \begin{defn}
Let $B_h \in \BD_{jk}$ and $P^{(h)}$ be a traversal of $B_h$, $\forall h = 1, \ldots, H$. We call $P^{(h)}$'s $j$-sections with number $H$ and $B_h$'s the associated $j$-blocks if  $B_1, \ldots, B_H$ are $j$-separated.
 \end{defn}

In order to define the $k$-block-nest of a path, we will give the {\em induction operation} first. Suppose $B \in \BD_{(j+1)k}$, and $P$ is a traversal of $B$. We denote by $P_0, P_1, P_2, \ldots$ the sequence of vertices along the path $P$. Let $s_1 = 0$ and $B_h = {BD}_{jk} (P_{s_h})$, where $BD_{jk} (z)$ is the unique element in $\BD_{jk}$ containing $z$. Let $s_{h+1}$ be the first time $P$ departs $\cup_{1\leq i\leq h}B^{**}_i$ forever, i.e,
 \begin{equation} \label{Eq.sh}
s_{h+1} = \inf \{ s \ge s_h + 1 : P_r \notin \cup_{1\leq i\leq h}B^{**}_i , \mbox{ for all } r \ge s \}\,,
 \end{equation}
where we use the convention that $\inf \emptyset = \infty$. For $h=1, 2, \ldots$, define
 $$
t_h = \inf \{ r \ge s_h + 1 : P_r \in \partial B^*_h \}
 $$
if $s_h < \infty$, and $t_h = \infty$ otherwise.

\begin{figure}[h]
\hspace{1cm}  \includegraphics[width=15cm]{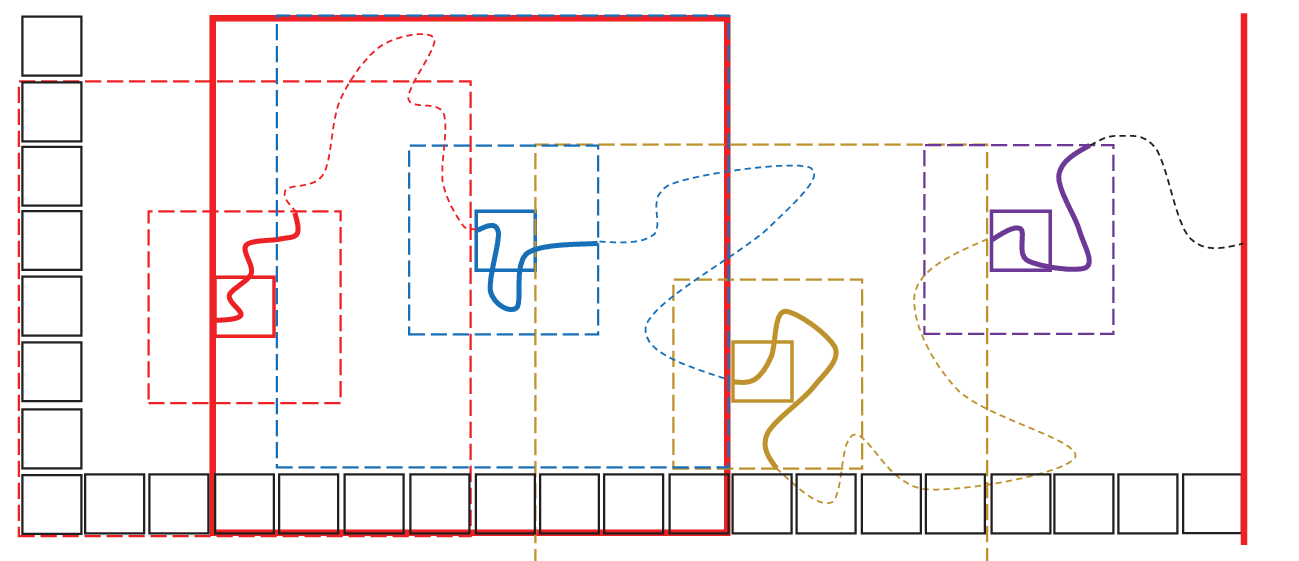}
\\ \vspace{-0.5cm}  \caption{ $k=3$ and $K=8$. The biggest red box is $B$, and the rightmost red line lies in $\partial B^*$. The smallest boxes are in $\BD_{jk}$. The curve stands for $P$, where the solid parts are $P|_{[s_h, t_h]}$'s with the red one corresponding to $h=1$, and the dashed parts will be removed to obtain $Q$. The smallest colorful boxes are $B_h$'s, and the colorful boxes with dashed lines are $B_h^*$'s and $B_h^{**}$'s. }\label{G.blocknest} \end{figure}

We denote by $\tau$ the first $h$ so that $t_{h+1} = \infty$. Since $B_1, B_2, \ldots, B_\tau$ are coherent and $P$ has distance $\ge 2^{(j+1)k}$, we see that $|P_{s_{h+1}} - P_{s_h}| \le 4 \times 2^{jk}$ and furthermore $\tau$ is no less than
 $$
K_1 = \frac {2^{(j+1)k}}{4 \times 2^{jk}}  = 2^{k-2} .
 $$
Note $t_h < s_{h+1}$. Then, we can obtain $j$-sections $P|_{[s_h,t_h]}$, $h=1,\cdots, K_1$ from $P$, as well as the associated blocks $B_h$'s, which are coherent (and disregard the boxes $B_{K_1 + 1}, \ldots, B_\tau$). Notice that $B_1$ is on the boundary of $B$, which means $\partial B_1 \cap \partial B \neq \emptyset$. In general, suppose that $P$ consists of $(j+1)$-sections with number $H$. We then deal with each section of $P$ as above, and obtain $H$ families of $j$-sections as well as their associated blocks. Since different sections of $P$ have distances larger than $2^{(j+1)k}$, $j$-sections and associated $j$-blocks in different families have distances larger than $2^{jk}$. Finally, combining all families together, we obtain $j$-sections with number $H K_1$. In addition, the associated blocks $B_{(r-1)K_1 +1}, \ldots, B_{rK_1}$ (in each family) are coherent for $r = 1, \ldots, H$.

Now we are ready to construct the $k$-block-nest of a path $P$ in $V_N$ with distance at least $N^{1- \k}$. Denote
 $$
j_0 =  \lfloor (1 - \frac 1 2 \d^2) m \rfloor\,.
 $$
Then $N^{1-\k } > 4 \times 2^{j_0 k}$ since $\k < \frac 1 2 \d^2$ and $n$ is large enough. Consequently, we can delete the parts of $P$ before it hits the boundary of $B_{j_0, 1} = BD_{j_0 k} (P_0)$ and after it hits $\partial B^*_{j_0,1}$, where $P_0$ is the starting point of $P$. Then we obtain a traversal of $B_{j_0,1}$, denoted by $P^{(j_0)}$. By the induction operation, for each $j = j_0-1, \ldots, 1$ we obtain a family of $j$-sections  $\mathcal P^{(j)}$ and the associated $j$-blocks $\{B_{j,h} , h = 1, \ldots, K_1^{j_0 - j} \}$. We call $\mathcal P^{(j)}$ the family of $j$-sections of $P$, and $\{B_{j,h}\}$ the family of $j$-blocks of $P$.

Next, we are going to bound the number of possible families of $j$-blocks. Note $B_{j_0, 1}$ has $2^{2(m+1-j_0) k}$ possible choices. Note $B_{j, (r-1)K_1 + 1}$ is on the boundary of $B_{j+1,r}$ by the induction operation. Given each $B_{j+1, r}$, there exist at most  $4 ( 2^k - 1) $ possible choices of $B_{j, (r-1)K_1 + 1}$. By the coherent property, there are at most $32$ possible choices of $B_{j,h+1}$ for each $B_{j,h}$, $h=(r-1) K_1 +1, \ldots, r K_1-1$. Thus, there are at most
 $$
4 \times 2^k \times 32^{K_1-1} \le 2^{2+k+5(K_1 - 1)} \le 2^{5 K_1 + k} =: c
 $$
possible families of coherent blocks $B_{j, (r-1)K_1 + 1}, \ldots, B_{j, rK_1}$, given $B_{j+1,r}$. By induction, for $j=j_0-1, \ldots, 1$ there are at most
 $$
2^{2(m+1-j_0)k} c^1 c^{K_1}  c^{K_1^2} \ldots c^{K_1^{j_0-1-j}} \le 2^{2(m+1-j_0)k} c^{\frac {K_1^{j_0-j}}{K_1 - 1}} = 2^{2(m+1-j_0)k + \frac {5 K_1 + k}{K_1 - 1} {K_1^{j_0 - j}} } \le 2^{2(m+1-j_0) k + 6 K_1^{j_0 - j}}
 $$
possible families of $j$-blocks, where the last inequality holds for $k \ge 6$. The above upper bound is also valid for $j = j_0$.

At last, we will set $j=0$ and give $Q$. For each $1$-block $B$, we wish to use the induction operation introduced above to extract the part $Q \cap B^*$ along the corresponding $1$-section, which is a traversal of $B$. Note each $B_h$ in the induction operation is a single-point set since $j=0$ now. In order to allow us to take advantage of the $q$-dependence of the Bernoulli process $\{ \xi_{N,z}, z \in V_N \}$ later, we replace $B_i^{**}$ with $\{ w : d (w, B_i) \le q \}$ in \eqref{Eq.sh}. This is to ensure each pair of the $s_h$'s defined in \eqref{Eq.sh} has distance at least $q+1$, so that $\xi_{N,w}$, $w \in Q$ are independent. Then, we obtain
 $$
K_2 : = \lfloor \frac {2^k}{q + 1} \rfloor
 $$
points of $s_h$'s in $B^*$, and they compose $Q \cap B^*$. We call $Q$ the {\em output} of $P$ via the $k$-block-nest program. Similar reasoning as above implies that given a $1$-block $B$, there are at most
 $$
4 \times 2^k \times \big( 4 (2q+1) \big)^{K_2-1} \le 2^k (8q+4)^{K_2} = : \tilde c
 $$
possible $Q \cap B^*$. Thus there are at most
 $$
2^{2(m+1-j_0) k + 6 K_1^{j_0 - 1}} \times \tilde c^{K_1^{j_0-1}} =   2^{2(m+1-j_0) k + (k+6)K_1^{j_0-1}} (8q+4)^{K_1^{j_0 - 1} K_2} \le 2^{2(m-j_0) k} (8q + 5)^{K_1^{j_0 - 1} K_2}
 $$
possible outputs for $k \ge k^\prime (q)$, where $k^\prime(q)$ is taken to ensure $2^{\frac {k+8}{K_2}} \le \frac {8q+5}{8q+4}$ and $K_1, K_2 \ge 1$. Furthermore, we take $k^\prime (q) > \log_2 q$ to ensure that the points in the output $Q$ have distances at least $q+1$ to each other. Then, there exists $k^\prime (q)$ such that following propositions hold.

 \begin{prop} \label{blocknest}
Let $k \ge k^\prime (q)$. Suppose $P$ is a path of distance at least $N^{1-\k }$, $Q$ is the output of $P$, and $B_{j,h}$, $1 \le h \le K_1^{j_0-j}$ are the $j$-blocks of $P$, $\forall j=1, \ldots, j_0$. Then, the following properties hold.

\noindent a) For $z,w \in Q$, $z \neq w$, we have $|z-w| > q$.

\noindent b) For $1 \le j \le j_0 - 1$ and $h \neq h^\prime$, we have $d ( B^*_{j, h}, B^*_{j,h^\prime} ) > 2^{jk}$.

\noindent c)  $|Q| = K_1^{j_0-1} K_2 $, $Q = \cup_{h=1}^{K_1^{j_0-j}} ( Q \cap B^*_{j,h})$, and $|Q \cap B^*_{j,h}| =K_1^{j-1} K_2$, $\forall h$.
 \end{prop}
 \begin{prop} \label{jkblocks}
Let $k \ge k^\prime(q)$. For all paths with distances at least $N^{1-\k }$, there are at most $2^{2(m+1-j_0) k + 6 K_1^{j_0 - j}}$ possible families of $j$-blocks, $\forall j = 1£¬ \ldots, j_0$, and at most $ 2^{2(m-j_0) k} (8q+5)^{K_1^{j_0 - 1} K_2}$ possible outputs.
 \end{prop}

\subsection{Proof of Theorem~\ref{cardinality}} \label{S.pfofcard}

Since every path of diameter at least $N^{1-\kappa}$ contains a subpath of distance at least $N^{1-\kappa}$, we can assume without loss of generality that $P$ has distance at least $N^{1-\kappa}$. Let $k \ge k^\prime(q)$ and $n$ large enough such that the $k$-block-nest program works, and Propositions~\ref{blocknest} and \ref{jkblocks} hold. Let $Q$ be the $k$-output of $P$. Denote
 $$
\tilde Q : = \{ w \in Q : \varphi_{N,w}  \le 4 \d \log N \} , \ \ \
\hat Q: = \{ w \in Q : \xi_{N,w} = 1 \} .
 $$
We will show the following lemmas for $k$ larger than some $\tilde k (\d, q)$ and $p$ larger than some $p_0 (\d, q)$.
 \begin{lemma} \label{Lemma.Xgood}
$\P (|\tilde Q| \ge \frac \d {3 + 4\d} |Q| \mbox{ for all possible outputs } Q ) \ge 1 - 5 e^{-n/10}$ for $n$ large enough.
 \end{lemma}
 \begin{lemma} \label{Lemma.Ygood}
$\P (|\hat Q| \ge (1 - \frac \d {6+8\d}) |Q| \mbox{ for all possible outputs } Q )  \ge 1 - e^{-n/10} $ for $n$ large enough.
 \end{lemma}
\noindent Assuming these, we obtain that
 $$
\P ( |\tilde Q \cap \hat Q| \ge \frac \d {6 + 8 \d} |Q| \mbox{ for all possible outputs } Q ) \ge 1 - 6 e^{-n/10} .
 $$
Thus, in order to prove Theorem~\ref{cardinality}, it suffices to check that $\frac \d {6 + 8 \d} |Q| \ge N^{1-\d} $. By Proposition~\ref{blocknest}, $|Q|=K_1^{j_0 - 1} K_2 \ge K_1^{j_0 - 1}$, where $K_1 = 2^{k-2}$ and $j_0 = \lfloor (1 - \frac 1 2 \d^2) m \rfloor$. This implies that
 $$
\frac {N^{1-\d}}{|Q|} \le 2^{(1-\d)k(m+1) - (k-2) \big( (1 - \frac {\d^2} 2) m - 2 \big) }  = 2^{k(1-\d)+ 2(k-2) } \( 2^{- (\d - \frac {\d^2} 2)k + 2 - \d^2 } \) ^m.
 $$
Take $k_0 : = k_0 (\d, q) \ge k^\prime(q) \vee \tilde k (\d, q)$ such that $(\d - \frac {\d^2} 2) k_0 - 2  + \d^2 > 0$. Then for all $k \ge k_0$ and $p \ge p_0$, we have $\frac {N^{1-\d}} {|Q|} \le \frac \d {6 + 8 \d } $ provided that $n$ is sufficiently large. This completes the proof of Theorem~\ref{cardinality}.

It remains to provide proofs for Lemmas~\ref{Lemma.Xgood} and \ref{Lemma.Ygood}.
\begin{proof}[Proof of Lemma~\ref{Lemma.Xgood}] It is clear from a simple union bound that
$$
\P \( \varphi_{N,z} \ge -  3 \log N, \ \ \forall z \in V_N \) \ge 1 - 2 e^{- n/10}.
$$
Together with Lemma \ref{d2mtom}, we can suppose without loss of generality that $\varphi_{N,w} \ge - 3 \log N$ and $\theta_w = \sum_{j= \lceil (1-\d^2) m \rceil}^{m-1} \psi_{j,w} \le \frac 5 2 \d \log N$ for all $w \in V_N$ (holding with probability $\ge 1 - 4 e^{-n/10}$).  Thus, we have
  \begin{eqnarray*}
\frac 1 {|Q|} \sum_{w \in Q} \varphi_{N,w}
 & \ge &
\frac {|\tilde Q|} {|Q|}  (- 3 \log N) + \frac {|Q| - |\tilde Q|} {|Q|} \times 4 \d \log N =
\( 4 \d - (3 + 4 \d)  \tfrac {|\tilde Q|}{|Q|} \) \log N .
 \end{eqnarray*}
Suppose $|\tilde Q| \le \frac {\d}{3 + 4\d} |Q|$, i.e., $(3+4\d) \frac {|\tilde Q|}{|Q|} \le \d$. Then
 $$
\frac 1 {|Q|} \sum_{w \in Q} \varphi_{N,w} \ge 3 \d \log N .
 $$
This, together with $\frac 1 {|Q|} \sum_{w \in Q} \theta_w \le \max_{z \in V_N} \theta_z \le \frac 5 2 \d \log N$, yields that (recall \eqref{eq-def-psi})
 $$
\frac 1 {|Q|} \sum_{w \in Q}   \sum_{j=0}^{\lceil (1 - \d^2) m \rceil - 1} \psi_{j,w}  \ge \frac 1 2 \d \log N \ge \frac 1 2 \d k \log 2 \times m .
 $$
Hence, there exists $0 \le j \le \lceil (1 - \d^2) m \rceil - 1$ such that
\begin{equation}\label{eq-to-check}
\frac 1 {|Q|} \sum_{w \in Q} \psi_{j,w}  \ge \frac 1 2 \d k \log 2 .
 \end{equation}
It remains to check that the probability for \eqref{eq-to-check} to hold for some output $Q$ is at most $e^{-n/10}$.

For $j \ge 1$, let $\{ B_{j,h} : h=1, \cdots, K_1^{j_0-j} \}$ be the family of $j$-blocks of $P$. By Proposition \ref{blocknest}, $Q$ can be decomposed into a disjoint union of $Q \cap B^*_{j,h}$'s, each of which has the same cardinality. Then, \eqref{eq-to-check} implies that
 $$
\frac 1 {K_1^{j_0-j}} \sum_{h=1}^{K_1^{j_0-j}} T_{j,h} \ge \frac 1 2 \d k \log 2 ,
 $$
where $T_{j,h} : = \max_{z \in B^*_{j,h}} \psi_{j,z}$.

We deal with a specific family of $T_{j,h}$'s first. They are independent and identically distributed, since $d(B^*_{j,h}, B^*_{j,h^\prime} ) > 2^{jk}$ by Proposition~\ref{blocknest}. It follows that for all  $a > 0$,
 $$
\P \( \tfrac 1 {K_1^{j_0-j}} \mbox{$\sum_{h=1}^{K_1^{j_0-j}}$}  T_{j,h} \ge \frac 1 2 \d k \log 2  \) \le \E \exp \( a  \mbox{$\sum_{h=1}^{K_1^{j_0-j}}$} ( T_{j,h}  - \frac 1 2 \d k \log 2 ) \) = \( \E e^{a(T  - \frac 1 2 \d k \log 2) } \) ^{K_1^{j_0-j}},
 $$
where $T = T_{j,1}$. By Lemma \ref{Emax}, $\E T \le 4 C_F \sqrt k$.  Combined with Lemma \ref{expmean}, it follows that for all $a > 1/\sqrt{2 \pi k}$,
  \begin{eqnarray*}
\E e^{a ( T  - \frac 1 2 \d k \log 2 )}
 & = &
\E e^{a ( T - \E T  )} \times e^{ a ( \E T - \frac 1 2 \d k \log 2) }
 \\ & \le &
3 \sqrt{2 \pi ka^2} e^{\frac 1 2 ka^2 } \times e^{a( 4 C_F \sqrt k - \frac 1 2 \d k \log 2) }
 \\ & = &
3 \sqrt{2 \pi ka^2} e^{\frac 1 2 ka (8 C_F / \sqrt k + a - \d \log 2 ) }.
 \end{eqnarray*}
Especially, we aim to set $a = \frac 1 2 \d \log 2$. There exists $k_2 : = k_2 (\d, q) \ge k^\prime(q)$ such that $a : = \frac 1 2 \d \log 2 > 1 / \sqrt{2 \pi k}$ and $8 C_F / \sqrt k + a - \d \log 2  \le - \frac 1 4 \d \log 2 $ for all $k \ge k_2$. It follows that
 \begin{eqnarray*}
\E e^{a ( T  - \frac 1 2 \d k \log 2 )}
 & \le &
3 \sqrt{2 \pi k} \times \frac 1 2 \d \log 2 \times e^{ -\frac 1 2 k \times \frac 1 2 \d \log 2 \times \frac 1 4 \d \log 2}
 \\ & = &
\frac 3 2 \sqrt{2 \pi k} \d \log 2 \times e^{- (\frac 1 4 \d \log 2)^2 k} .
 \end{eqnarray*}
Take $k_3 : = k_3 (\d, q) \ge k_2$ such that the right hand side above is less than $e^{- \frac 1 {36} \d^2 k} $ for all $k \ge  k_3$. Consequently,
 $$
\P \( \tfrac 1 {K_1^{j_0-j}}  \mbox{$\sum_{h=1}^{K_1^{j_0-j}}$}  T_{j,h} \ge \frac 1 2 \d k \log 2  \)  \le e^{- \frac 1 {36} \d^2 k K_1^{j_0-j} }
 $$
for any specific family of $T_{j,h}$'s. By Proposition~\ref{jkblocks}, there are at most $2^{2(m+1-j_0) k + 6 K_1^{j_0 - j}}$ possible families of $T_{j,h}$'s. Therefore, for each fixed $j$,
  \begin{align*}
\P \( \exists Q \mbox{ such that } \tfrac 1 {K_1^{j_0-j}} \mbox{$\sum_{h=1}^{K_1^{j_0-j}}$} T_{j,h}  \ge \frac 1 2 \d k \log 2 \)
&\leq
2^{2(m+1-j_0) k + 6 K_1^{j_0 - j}} e^{- \frac 1 {36} \d^2 k K_1^{j_0-j} } \\
&\leq
2^{2(m+1-j_0) k} \( 2^6  e^{- \frac 1 {36} \d^2 k  } \) ^{K_1^{j_0-j}}.
 \end{align*}
It follows that
 \begin{eqnarray*}
 & &
\P \( \exists Q \mbox{ such that }  \tfrac 1 {K_1^{j_0-j}} \mbox{$ \sum_{h=1}^{K_1^{j_0-j}}$} T_{j,h}  \ge \frac 1 2 \d k \log 2 \mbox{ for some } j = 1, \ldots, \lceil (1 - \d^2) m \rceil - 1  \)
 \\ & \le &
\sum_{j = 1}^{\lceil (1 - \d^2) m \rceil - 1} 2^{2(m+1-j_0) k} \( 2^6  e^{- \frac 1 {36} \d^2 k  } \) ^{K_1^{j_0-j}}
  \\ & = &
2^{2(m+1-j_0) k} \sum_{j = j_0 -\lceil (1 - \d^2) m \rceil + 1}^{j_0 - 1} \( 2^6 e^{- \frac 1 {36} \d^2 k} \) ^{K_1^{j}}
 \\ & \le &
2^{2(m+1-j_0) k} \( 2^6 e^{- \frac 1 {36} \d^2 k} \) ^{K_1^{j_0 - \lceil (1 - \d^2) m \rceil + 1}} / ( 1 - 2^6 e^{- \frac 1 {36} \d^2 k} )  .
 \end{eqnarray*}
Recall $j_0 = \lfloor (1 - \frac 1 2 \d^2) m \rfloor$, which implies $j_0 - \lceil (1 - \d^2) m \rceil + 1 \ge \frac 1 2 \d^2 m - 1$. Thus $K_1^{j_0 - \lceil (1 - \d^2) m \rceil + 1 } \ge 2^{\frac 1 4 \d^2 n}$. Therefore, the right hand side above is less than  $\frac 1 2 e^{-n/10 }$  provided that $k\geq k_4 : = k_4 (\d, q) (\ge k_3)$ and $n$ is large enough.

For $j=0$, we denote $\{ \psi_{0,w}, w \in Q \}$ by $\{ T_{0,h}, h = 1, \cdots, |Q| \}$. By a) of Proposition~\ref{blocknest}, they are independent and distributed as $N(0,k)$. Then, a union bound combined with Proposition~\ref{jkblocks} and Lemma~\ref{concentration} implies that
 \begin{eqnarray*}
\P \( \exists Q \mbox{ such that } \frac 1 {|Q|} \sum_{w \in Q} \psi_{0,w}
\ge \frac 1 2 \d k \log 2 \)
  & \le &
2^{2(m-j_0) k} (8q+5)^{K_1^{j_0 - 1} K_2} 2 e^{-\frac 1 8 k (\d \log 2)^2 K_1^{j_0-1}K_2}
 \\ & \le &
2^{2(m-j_0) k + 1} \(  (8q+5) e^{-\frac 1 8 k (\d \log 2)^2  } \) ^{ K_1^{j_0-1} K_2} .
 \end{eqnarray*}
Note $K_2 = \lfloor \frac {2^k} {q + 1} \rfloor$. The right hand side above is less than $\frac 1 2 e^{-n/10 }$ provided $k \ge \tilde k (\d, q) (\ge k_4)$. 

Therefore, the event in \eqref{eq-to-check} happens for some output $Q$ with probability at most $e^{-n/10}$, completing the proof.
\end{proof}

 \begin{proof}[Proof of Lemma~\ref{Lemma.Ygood}]
Investigate a specific output $Q$ first. By Proposition~\ref{blocknest}, $\xi_{N,w}$, $w \in Q$ are independent. Denote $\tilde p = 1 - \frac {\d}{6 + 8 \d} $. For $p > \tilde p$, by Chebyshev's inequality, $\forall a > 0$,
 $$
\P ( |\hat Q| \le \tilde p |Q| ) \le \prod_{w \in Q} \E  e^{a (\tilde p - \xi_{N,w})}  \le \( \E e^{a (\tilde p - \xi)} \) ^{|Q|} ,
 $$
where $\P (\xi = 1 ) = 1 - \P (\xi = 0) = p$.
Take $a$ such that $e^a = \frac { p (1 - \tilde p ) }{\tilde p (1-p)}$. Then $\E e^{a (\tilde p -\xi)} = e^{-\rho}$, where
 $$
\rho = \rho (\d) : = f (\tilde p) -  f (p) + (p - \tilde p) f^\prime (p) > 0,
 $$
and $f(x) := x \log x + (1 - x) \log (1-x)$. Note there are at most $2^{2(m-j_0) k} (8q+5)^{K_1^{j_0 - 1} K_2}$ possible outputs. Therefore,
 \begin{eqnarray*}
 & &
\P ( |\hat Q| \le  \tilde p |Q| \mbox{ for some output } Q )  \le e^{- \rho K_1^{j_0-1} K_2 } \times 2^{2(m-j_0) k} (8q+5)^{K_1^{j_0 - 1} K_2}
 \\ & \le &
2^{2(m-j_0)k} \( e^{-\rho} (8q+5) \) ^{K_1^{j_0 - 1} K_2} .
 \end{eqnarray*}
Note $\rho$ increases to $\infty$ as $p$ increases to $1$, i.e. $e^\rho > 8q+5$ for $p$ being larger than some $p_0 (\d, q) \in (\tilde p, 1)$. Consequently, the right hand side above is less than $e^{- n /10}$ for $n$ large enough.
\end{proof}

\section{Proof of Theorem \ref{maintheorem}} \label{pfofthm}

In Section~\ref{S.largedist}, we show that with high probability a typical pair of vertices sampled from $\mu_\gamma \times \mu_\gamma$ has $\ell_\infty$-distance at least $N^{1-\k}$. Therefore, the lower bound on their FPP distances follows readily from Theorem~\ref{cardinality}. Much of the work goes into the proof of the upper bound, for which we need to convert boundary-to-boundary crossings to vertex-to-vertex crossings. The proof technique is a Russo-Seymour-Welsh kind of construction \cite{R,SW}.

\subsection{Macroscopic $\ell_\infty$-distance for a typical pair} \label{S.largedist}

The goal of this section is to prove the following statement.
\begin{prop} \label{shortdist}
For any fixed $\kappa > 0$, there exists $\rho > 0$ such that
 $$
\lim_{ N\to \infty} \P \( \mu_\gamma\times \mu_\gamma ( \big\{ (u, v): |u-v| < N^{1-\k } \big\} ) > N^{-\rho}  \) = 0 .
 $$
 \end{prop}
 By definition, we have
 \begin{equation}\label{eq-ell-1-distance}
\mu_\gamma\times \mu_\gamma ( \{(u, v): |u-v| < N^{1-\k} \} ) = \frac {\sum_{|u-v| < N^{1-\k} } e^{\gamma (\varphi_{N,u} + \varphi_{N,v})}}{\( \sum_w e^{\gamma \varphi_{N,w}} \) ^2}\,.
 \end{equation}
Thus, in order to prove Proposition~\ref{shortdist}, we need to control both the numerator and the denominator in \eqref{eq-ell-1-distance}. We start with the denominator.

 \begin{lemma} \label{denominator}
For any $\tau > 0$,
 $$
\lim_{N \to \infty} \P \( \mbox{$ \sum_{w \in V_N} $} e^{\gamma \varphi_{N, w}}  \ge N^2 e^{\frac 1 2 \gamma^2 n} e^{- \tau n} \) = 1 .
 $$
 \end{lemma}
\begin{proof}
The statement corresponds to an estimate on the cardinalities for the level sets of the Gaussian field, which has been well-understood for log-correlated Gaussian fields. For instance, the dimension of level sets for the Gaussian free field was computed in \cite{Daviaud}. While the technique in \cite{Daviaud} extends to our $K$-coarse MBRW, we choose to apply a more handy theorem from \cite{CDD13}.

Take $R >  2 \vee \sqrt {\frac 4 \tau \log 2}$. Since $\gamma < 2 \sqrt{\log 2}$, one can take $1 \le r \le R-1$ such that $| \frac r R - \frac {\gamma } {2 \sqrt{ \log 2}} | < \frac 1 R$. By Lemma \ref{meanofmax},  we have  $\E \max_{z \in V_N} \varphi_{N,z} \ge 2 \sqrt{\log 2} n - \frac 1 { \sqrt{\log 2}} \log n$, which implies that $\{ \varphi_{N,w} : w \in V_N \}$ is an \emph{extremal} (using the notion of \cite{CDD13}) field. Thus, by \cite[Theorem 1.6]{CDD13}, with  probability tending to 1 as $N \to \infty$ the level set
 $$
U_{N, \frac r R} : = \{ w \in V_N : \varphi_{N,w} \ge \frac r R \E \max_{z \in V_N} \varphi_{N,z} \}
 $$
has cardinality at least $N^{2 (1 -  \frac {r^2}{R^2} - \frac 1 {R^2} )}$. On this event, it then follows that
 \begin{eqnarray*}
\sum_{w \in V_N}  e^{\gamma \varphi_{N, w}}
  & \ge &
\sum_{w \in U_{N, \frac r R}}  e^{\gamma \varphi_{N, w}} \ge   N^{2 (1 - \frac {r^2}{R^2} - \frac 1 {R^2})}  e^{\gamma \frac {r}{R} ( 2 \sqrt{\log 2} n - \frac 1 {\sqrt{\log 2}} \log n ) }
  \\ & \ge &
n^{- \frac 1 {\sqrt{\log 2}} \gamma} N^{2 (1 - \frac 1 {R^2})} e^{ n ( \frac {r}{R} 2\gamma \sqrt{\log 2} -  \frac {r^2}{R^2} 2 \log 2) }
 \\ & \ge &
n^{- \frac 1 {\sqrt{\log 2}} \gamma} N^{2 (1 - \frac 1 {R^2})} e^{- \frac {2 \log 2}{R^2} n } e^{\frac 1 2 \gamma^2 n}\,,
 \end{eqnarray*}
where in the last inequality we use
 $$
\frac r R 2 \gamma \sqrt{\log 2} - \frac {r^2} {R^2} 2 \log 2 =  -2 \log 2 ( \frac r R- \frac {\gamma}{2 \sqrt{\log 2}} )^2 + \frac {1}{2} \gamma^2 \ge -  \frac {2 \log 2 } {R^2}  + \frac 1 2 \gamma^2 \,.
 $$
Recalling that $\tau > \frac 4 {R^2} \log 2$, we obtain that $\sum_{w \in V_N}  e^{\gamma \varphi_{N, w}} \geq N^2 e^{\frac 1 2 \gamma^2 n } e^{- \tau n}$ as desired.
\end{proof}
\begin{proof}[Proof of Proposition~\ref{shortdist}]
In light of \eqref{eq-ell-1-distance} and Lemma~\ref{denominator}, it remains to bound the numerator in \eqref{eq-ell-1-distance} from above. For this purpose, we apply a straightforward first moment calculation. Note that
 $$
\E e^{\gamma (\varphi_{N,u} + \varphi_{N,v})} = e^{\frac 1 2 \gamma^2 \Var (\varphi_{N,u} + \varphi_{N,v})} = e^{\gamma^2 (km + \E \varphi_{N,u} \varphi_{N,v})}\,.
 $$
By Lemma \ref{covariance}, we have that $\E \varphi_{N,u} \varphi_{N,v} \le n - \log_2 (|u-v| \vee 1)$.  Consequently,
 \begin{equation} \label{expmoment}
\E e^{\gamma (\varphi_{N, u} + \varphi_{N, v})} \le e^{(2n   - \log_2 (|u-v|\vee 1) )\gamma^2 } \le e^{2n \gamma^2} e^{- \gamma^2 \lfloor \log_2 (|u-v| \vee 1) \rfloor } .
 \end{equation}
Grouping pairs of vertices in terms of a dyadic decomposition for their $\ell_\infty$-distance and summing over all possible groups, we get that
 $$
\E  \sum_{u\neq v, |u-v| < N^{1-\k }} e^{\gamma (\varphi_{N,u} + \varphi_{N,v})}
  \le
e^{2n \gamma^2} \sum_{r} e^{- \gamma^2 r} \left| \left\{ (u,v) : 2^r \le |u-v| < 2^{r+1}, |u-v| < N^{1-\k }  \right\} \right| .
 $$
Note that $| \{ (u,v) : 2^r \le |u-v| < 2^{r+1} \} |\leq 16 N^2 2^{2r}$, and $r < n (1-\k )$. Thus,
 \begin{eqnarray*}
 & &
\E  \sum_{u \neq v, |u-v| < N^{1-\k }} e^{\gamma (\varphi_{N,u} + \varphi_{N,v})}
 \\ & \le &
e^{2n  \gamma^2} 16  N^2  \sum_{r=0}^{ \lfloor n (1-\k ) \rfloor  } 2^{2r} e^{-\gamma^2 r}
 \le
16 e^{2n \gamma^2} N^2  \frac {(2^2 e^{-\gamma^2})^{\lfloor n (1-\k ) \rfloor  + 1 }  }{ 2^2 e^{-\gamma^2} - 1}
 \\ & \le &
\frac { 64 e^{- \gamma^2}   }{ 4 e^{-\gamma^2} - 1} e^{2n  \gamma^2} N^2  (2^2 e^{-\gamma^2})^{ n (1-\k ) }
 =
\frac {  64 e^{- \gamma^2} }{ 4 e^{-\gamma^2} - 1}N^{- 2 \k } e^{ \k  \gamma^2 n }  N^4 e^{ \gamma^2 n} ,
 \end{eqnarray*}
where $2^2 e^{-\gamma^2} > 1$ since $\gamma < \sqrt{\log 4}$. It follows that
 $$
\E  \sum_{|u-v| < N^{1-\k }} e^{\gamma (\varphi_{N,u} + \varphi_{N,v})}  \le N^2 e^{2 \gamma^2 n} + \frac {  64 e^{- \gamma^2} }{ 4 e^{-\gamma^2} - 1}N^{- 2 \k } e^{ \k  \gamma^2 n }  N^4 e^{ \gamma^2 n} .
 $$
Combined with Markov's inequality, it yields that for any fixed $\rho>0$,
 \begin{eqnarray*}
 &&
\P \( \mbox{$ \sum_{|u-v| < N^{1-\k }}$} e^{\gamma (\varphi_u + \varphi_v)}  > N^{-\rho}  e^{-2 \tau n} N^4 e^{\gamma^2 n}  \)
 \\ & \le &
\frac {N^{\rho}  e^{2 \tau n} \E \sum_{|u-v| < N^{1-\k }} e^{\gamma (\varphi_u + \varphi_v)} }{  N^4 e^{\gamma^2 n}  }
 \le
N^{\rho} e^{2 \tau n} \( \frac {e^{\gamma^2 n}}{N^2 }  + \frac {  64 e^{- \gamma^2} }{ 4 e^{-\gamma^2} - 1} \times  \( \frac { e^{\gamma^2 n } } {N^2} \) ^{\kappa} \) \,,
 \end{eqnarray*}
which tends to 0 as $n\to \infty$ provided that $\gamma < \sqrt{2 \log 2}$ and $\tau, \rho$ are small enough. Combined with Lemma~\ref{denominator}, this completes the proof of the proposition.
\end{proof}

\subsection{Lower bound for Theorem \ref{maintheorem}}

Set $p=1$ and $\d = \frac {\e} {1 + 4 \gamma} $. By Theorem \ref{cardinality} and symmetry, with probability at least $1 - 6 e^{- n/10 }$, we have that  $ | \{ w \in P: \varphi_{N, w} \ge - 4 \delta \log N \} | \ge N^{1 -  \d}$ for any path $P$ with diameter at least $N^{1-\k }$. In what follows, we assume without loss that this event occurs. Thus, for any pair $u, v$ with $|u-v| \ge N^{1-\k }$ we have that
 $$
d_\gamma (u,v) \ge N^{1-\d} \times e^{\gamma (-4 \d \log N)} = N^{1 - (1+4 \gamma)\d} = N^{1-\e}\,.
 $$

Therefore, we get that
 $$
\mu_\gamma\times \mu_\gamma \( \big\{ (u, v):   d_{\gamma}(u, v) \geq N^{1 - \epsilon} \big\} \) \ge \mu_\gamma\times \mu_\gamma \( \big\{ (u, v): |u-v| \ge N^{1-\k  } \big\} \) .
 $$
Combined with Proposition~\ref{shortdist}, this yields that for some $\rho>0$
$$
\lim_{N\to \infty} \P(\mu_\gamma\times \mu_\gamma (\{(u, v): N^{1 - \epsilon} \leq d_{\gamma}(u, v)\} ) \geq 1 - N^{-\rho}) = 1,
 $$
 completing the proof of the lower bound of Theorem~\ref{maintheorem}.

\subsection{Upper bound for Theorem \ref{maintheorem}}

This section is devoted to the upper bound of $d_\gamma (u,v)$. Namely, we prove
 $$
\lim_{N\to \infty} \P \( \mu_\gamma\times \mu_\gamma \( \big\{ (u, v): d_{\gamma}(u, v) > N^{1 + \e} \big\} \) > N^{-\rho} \) = 0.
 $$
 \noindent{\bf Outline of the proof.} We will show that for some fixed $\rho>0$,
 \begin{equation} \label{cgamma}
\lim_{N \to \infty} \P \( \mu_\gamma ( \{ u: \varphi_{N,u} > 2 \gamma \log N \} ) > N^{-\rho}\) = 0 \,.
 \end{equation}
Assume \eqref{cgamma}. By Proposition \ref{shortdist}, it suffices to prove for some fixed $\rho >0$ and a certain $\kappa>0$,
 $$
\lim_{N \to \infty}  \P \( \mu_\gamma\times \mu_\gamma \( \big\{ (u, v): |u-v| > N^{1 - \k }, \varphi_{N,u}, \varphi_{N,v} \le 2 \gamma \log N , d_{\gamma}(u, v) > N^{1 + \e}  \big \} \) > N^{-\rho}\) = 0 .
 $$
By Lemma \ref{denominator} and Markov's inequality, we only need to show for a certain $\tau>0$,
 \begin{equation}  \label{upperbound}
\lim_{N \to \infty} \frac { \sum_{|u-v| > N^{1-\k }} \E e^{\gamma (\varphi_{N,u} + \varphi_{N,v})} 1_{\{ \varphi_{N,u}, \varphi_{N,v} \le 2 \gamma \log N , \ d_\gamma (u,v) > N^{1 + \e} \} } }{ N^{-\rho} e^{ - 2 \tau n} N^4 e^{\gamma^2 n} } = 0\,.
 \end{equation}

To bound the numerator, we consider a pair $u, v$ with $|u-v| > N^{1- \k }$ and write
  \begin{eqnarray}
  & &
\E  e^{\gamma (\varphi_{N,u} + \varphi_{N,v})} 1_{ \{ \varphi_{N,u}, \varphi_{N,v} \le 2 \gamma \log N, \ d_\gamma (u,v) > N^{1 + \e} \} }  \nonumber
  \\ & = &
\E \( e^{\gamma (\varphi_{N,u} + \varphi_{N,v})} 1_{\{   \varphi_{N,u}, \varphi_{N,v} \le 2 \gamma \log N \}}  \P \( d_\gamma (u,v) > N^{1 + \e} | \varphi_{N,u}, \varphi_{N,v} \) \) . \label{integral}
 \end{eqnarray}
Then, we will show that if $|u-v| > N^{1-\k }$, one can find a ``short" path $P$ connecting $u$,$v$, on which the Gaussian variables $\varphi_{N,z}$'s are all small (see Proposition~\ref{goodpath}). Hence, $d_\gamma (u,v) > N^{1+\e}$ with small probability. For convenience, an open path $P$ is called \emph{good} if $\varphi_{N, z} \le 7 \d \log N$ for all $z \in P$. The following proposition on the existence of a good path between a typical pair is a key ingredient.
 \begin{prop} \label{goodpath}
Suppose $0 < \d < 1/2$, $0< \k < \frac 1 4 \d^2$, $0< \b < \frac 1 4 \d^4 \log^2 2 $. Denote
 $$
\ell_0 = k \lfloor \d^2 m \rfloor, \ \ \ \ell_1 =   \lfloor \log_2 ( \lceil N^{1 - 2 \k } \rceil + 2^{\ell_0+ 1})  \rfloor - 1 .
 $$
Then, there exist $p_1 = p_1 (\d, q) \in (0,1)$ and $K_1 = K_1 (\d, q)$ such that the following holds for $K \ge K_1$, $p \ge p_1$ and $N$ large enough. If $|u-v| > N^{1-\k}$, with probability at least $1 - 128(1-p)^{\frac 1 {(2q+1)^2}} - e^{-\b n}$ one can find a good path $P$ connecting $u$, $v$ such that $|P| \le N^{1+2 \d}$. Furthermore, there exist $P^C$, $P^F$ and $Q^\ell$, $\ell = \ell_0 + 1, \ldots, \ell_1$ such that $P \subseteq P^C \cup P^F \cup (\cup_{\ell=\ell_0 + 1}^{\ell_1} Q^\ell )$, and the following (i) and (ii) hold.

\noindent (i) $|P^C| \le 40 N^{(1+\d)\d^2}$, $|P^F| \le N^{1+2\d}$, $|Q^\ell| \le 40 \times 2^{(1+\d) \ell}$, for all $\ell = \ell_0 + 1, \ldots, \ell_1$.

\noindent (ii) $\s_{z,u} + \s_{z,v}$ is less than $4 n \k$ for all $z \in P^F$, and it is less than  $n(1+\k) + 1 - \ell$ for all $z \in Q^\ell$, $\ell=\ell_0 +1, \ldots, \ell_1$. (Recall that $\s_{z, u} = \E \varphi_{N, z} \varphi_{N, u}$).
 \end{prop}

\begin{remark}
(a) The assumption that $|u - v| > N^{1-\kappa}$ is immaterial and merely for the convenience of reusing results obtained in previous sections. (b) The additional properties for the good path in Proposition \ref{goodpath} are for the purpose of controlling how conditioning on the values at the endpoints $u$ and $v$ would change the behavior of the Gaussian values on the path, where $P^C$ and $P^F$ respectively consist of parts close to and far away from the endpoints (see Fig~\ref{G.RSW}).
\end{remark}

Finally, we will set $p = 1$ and show that for $\varphi_{N,u}, \varphi_{N,v} \le 2 \gamma \log N$, the good path given in Proposition \ref{goodpath} leads to $d_\gamma (u,v) \le N^{1+\varepsilon}$. This is incorporated in the following proposition.
 \begin{prop} \label{givenendvalues}
Set $p = 1$. Suppose the assumptions in Proposition \ref{goodpath} hold. Then the following holds when  $0<\delta<\delta_0 (\e)$ for some $\delta_0 (\e) >0$. If $|u-v| \ge N^{1-\k }$ and $x, y \le 2 \gamma \log N$, we have
$$
\P \( d_\gamma (u,v) > N^{1 + \varepsilon} | \varphi_{N,u} = x, \varphi_{N,v} = y \) \le e^{- \b n}\,.
 $$
 \end{prop}
This together with  \eqref{integral} and \eqref{expmoment} implies that for $|u-v| \ge N^{1-\k}$,
 \begin{align*}
\E  e^{\gamma (\varphi_{N,u} + \varphi_{N,v})} 1_{ \{ \varphi_{N,u}, \varphi_{N,v} \le 2 \gamma \log N,  d_\gamma (u,v) > N^{1 + \e} \} }
 &\le
e^{- \b n} \E e^{\gamma (\varphi_{N,u} + \varphi_{N,v})} \\
&\le  e^{- \b n} e^{\gamma^2 (2n - \log_2 |u-v|)}
 \le
 e^{- \b n} e^{\gamma^2 n (1+\k )} .
 \end{align*}
It follows that
 \begin{eqnarray*}
 & &
\frac { \sum_{|u-v| > N^{1-\k }} \E e^{\gamma (\varphi_{N,u} + \varphi_{N,v})} 1_{\{ \varphi_{N,u}, \varphi_{N,v} \le 4 \gamma \log N , \ d_\gamma (u,v) > N^{1 + \e} \} } }{ N^{- \rho} e^{ - 2 \tau n} N^4 e^{\gamma^2 n} }
 \\ & \le &
\frac { N^4 e^{- \b n} e^{\gamma^2 n (1+\k )} } { N^{- \rho} e^{ - 2 \tau n} N^4 e^{\gamma^2 n} } =  e^{ - ( \b - \gamma^2 \k - 2 \tau - \rho \log 2) n} .
 \end{eqnarray*}
Therefore, \eqref{upperbound} follows for $\k < \b / \gamma^2$ and $\tau, \rho$ small enough, completing the proof.

The rest of the paper is devoted to the proofs of \eqref{cgamma}, Propositions~\ref{goodpath} and \ref{givenendvalues}.

\subsubsection{Proof of \eqref{cgamma}.}

Suppose $Z \sim N(0,1)$, $b > a > 0$. Then
 \begin{align*}
\E e^{a Z} 1_{\{ Z > b \}} &= \int_b^\infty e^{az} \frac 1 {\sqrt{2\pi}} e^{- \frac 1 2 z^2} d z = \int_b^\infty \frac 1 {\sqrt {2 \pi}} e^{- \frac 1 2 (z-a)^2} e^{\frac 1 2 a^2} d z \\
&= \P(Z > b-a) e^{\frac 1 2 a^2} \le 2 e^{-\frac 1 2 (b-a)^2} e^{\frac 1 2 a^2}.
 \end{align*}
It follows that
 \begin{eqnarray*}
\E e^{\gamma \varphi_{N,u}}  1_{\{ \varphi_{N, u} > 2 \gamma \log N \}}
 & \le &
\E e^{\gamma \sqrt {km} Z} 1_{\{ Z > \gamma \sqrt {km} 2 \log 2 \}}
 \\ & \le &
2 e^{- \frac 1 2 (2 \log 2 - 1)^2 \gamma^2 km} e^{\frac 1 2 \gamma^2 km}  \le 2 e^{- \frac 1 2 (2 \log 2 - 1)^2 \gamma^2 (n-k)} e^{\frac 1 2 \gamma^2 n} .
 \end{eqnarray*}
Therefore,
 \begin{eqnarray*}
\frac {\sum_{u \in V_N} \E e^{\gamma \varphi_{N,u}} 1_{ \{ \varphi_{N,u} > 2 \gamma \log N \}}  }{ N^{-\rho} N^2 e^{\frac 1 2 \gamma^2 n} e^{- \tau n}} \le 2 N^{\rho}  e^{- \frac 1 2 (2 \log 2 - 1)^2 \gamma^2 (n-k)} e^{\tau n}.
 \end{eqnarray*}
By Lemma \ref{denominator} and Markov's inequality, \eqref{cgamma} holds provided $\tau, \rho$ are sufficiently (but fixed) small.

\subsubsection{Proof of Proposition \ref{goodpath}} \label{S.goodpath}

In order to prove Proposition~\ref{goodpath}, we will first show that with high probability there exists a good crossing from the left boundary to the right boundary of a large box. Namely, we prove this for boxes of side length $L = 2^{\ell}$ for all $\ell \ge \ell_0$ (recall $\ell_0 = k \lfloor \d^2 m \rfloor$).

For any rectangle $B$, denote by $\LR (B)$ all paths in $B$ from the left boundary of $B$ to the right one, and by $\UD (B)$ all paths in $B$ from the top boundary of $B$ to the bottom one. If $B$ is a box of side length $L = 2^\ell$, we denote $\cup_{a=0}^3 (B +(a L, 0))$ by  $\acute{B}$. That is, $\acute{B}$ is a rectangle with width $4 L$ and height $L$, which shares the same left bottom corner with $B$. We call $Q$ a $B$-crossing if it consists of two good paths $Q^V \in \UD (B)$ and $Q^H \in \LR (\acute{B})$ with $|Q^V| \le 4 L^{1 + \d}$ and $|Q^H| \le 16 L^{1+\d}$.

 \begin{lemma} \label{onetimesfour}
Suppose $0<\delta<1/2$ and $B$ is a box of side length $L = 2^{\ell}$ with $\ell_0 \le \ell \le n-2k$. Then
 $$
\P \( \mbox{there exists a } B \mbox{-crossing} \) \ge 1 - 9 e^{- \b_1 \ell} ,
\mbox{ where } \b_1 = \frac {\log^2 2} 2 \d^2 \,.
 $$
 \end{lemma}
To prove Lemma \ref{onetimesfour}, the following lemma will be useful.
 \begin{lemma} \label{orderofXz}
Let $\tilde N = 2^{k \tilde m}$, where $1 \le \tilde m < m$. Suppose $\tilde V_{\tilde N}$ is a $\tilde N \times \tilde N$ box, satisfying $\tilde V_{\tilde N} \subset V_N$. Then $\E \max_{z \in \tilde V_{\tilde N}}  ( \varphi_{N,z} -  \varphi_{\tilde N,z} ) \le \sqrt{8 k} C_F$. Furthermore,
\begin{equation}\label{eq-deviationofXz}
\P \( \max_{z \in \tilde V_{\tilde N}}  ( \varphi_{N,z} -  \varphi_{\tilde N,z} ) > 2 \d \log N \) \le 2 e^{- \frac {\d^2 \log^2 2} {2k(m- \tilde m)} n^2 }.
\end{equation}
 \end{lemma}
\begin{proof} Denote
 $$
\theta_z := \varphi_{N,z} - \varphi_{\tilde N, z} = \mbox{$\sum_{j= \tilde m }^{m - 1}$} \psi_{j,z}.
 $$
Then $\theta_z \sim N(0, k(m-\tilde m))$. For $z, w \in \tilde V_{\tilde N}$, $|z-w| < 2^{k \tilde m}$, thus by \eqref{covofpsi},
 \begin{eqnarray*}
\E \theta_z \theta_w
 & = &
\sum_{j= \tilde m}^{m-1} \E \psi_{j,z} \psi_{j,w} \ge k \sum_{j=\tilde m}^{m-1}  \( 1 - 2 \frac {|z-w|} {2^{jk}} \)
 \\ & = &
k (m-\tilde m ) - 2 k |z-w| \sum_{j = \tilde m}^{m-1} 2^{-jk} = k (m-\tilde m) - \frac {4k}{2^{k \tilde m }} |z-w| .
 \end{eqnarray*}
It follows that
 $$
\E (\theta_z - \theta_w)^2 \le  \tfrac {8k}{2^{k \tilde m }} |z-w| = 8k \tfrac { |z-w| }{\tilde N} .
 $$
By Lemma \ref{maxinbox}, we have
 $$
\E \max_{z \in \tilde  V_{\tilde N}} \theta_z \le  \sqrt {8k} C_F .
 $$
At this point, \eqref{eq-deviationofXz} follows from the preceding inequality and an application of Lemma~\ref{concentration}.
\end{proof}

\begin{proof}[Proof of Lemma \ref{onetimesfour}]
Suppose $ kr \le \ell + 1 < k (r+1) $, where $\lfloor \d^2 m\rfloor \le r \le m-2$. Denote $\tilde N = 2^{k (r+1)}$. We would like to mention that as $m \to \infty$, one has $\ell_0 \to \infty$ hence $r \to \infty$.

Denote by $\UD_c (B)$ the collection of all cut sets that separate the top and bottom boundaries of $B$ (with respect to the induced graph on $B$), and by $\LR_c(\acute{B})$ the collection of all cut sets that separate the left and right boundaries of $\acute{B}$ (with respect to the induced graph on $\acute{B}$). We will first show
\begin{equation} \label{crosswhp}
 \P \( E \) \ge 1 - 9 e^{- \b_1 \ell} ,
 \end{equation}
where
 $$
E := \left\{ \big| \{ w \in \mathcal C: \varphi_{N, w}  \leq 7 \delta \log N, \ \xi_{N,w} = 1 \} \big| \ge \tilde N^{1-\d}/4, \mbox{ for all }\mathcal C \in \UD_c (B) \cup \LR_c (\acute{B}) \right\} .
 $$

Let $\tilde V_{\tilde N}$ be the box of side length $\tilde N$ and sharing the same left bottom corner with $B$. Then $\acute{B} \subseteq \tilde V_{\tilde N}$. Consider the Gaussian field $\{ \varphi_{\tilde N, z} : z \in \tilde V_{\tilde N} \}$.   By \eqref{eq-continuity}, the event
 $$
E_1 : = \left\{ | \varphi_{\tilde N, z} - \varphi_{\tilde N, w} | \le \d \log N \mbox{ for all } z, w \in \tilde V_{\tilde N}, z \sim w \right\}
 $$
happens with probability at least $1 - e^{- k(r+1)/ 10} \ge 1 - e^{- \zeta_1 \ell}$, where we recall that $\zeta_1  =\frac {\log^2 2} 2 \d^2< 1/10$. By Lemma \ref{orderofXz}, the event
 $$
E_2 : = \left\{ \max_{z \in \tilde V_{\tilde N}} (\varphi_{N,z} - \varphi_{\tilde N, z}) \le 2 \d \log N   \right\}
 $$
happens with probability $1- 2 e^{- \frac {\d^2 \log^2 2}{2k (m-r-1)} n^2} \ge 1 - 2 e^{- \frac {\d^2 \log^2 2 }2 n} \ge 1 - 2 e^{- \zeta_1 \ell} $. For any set $A$, denote
 $$
\Psi(A) : = \left\{ v \in A : \varphi_{\tilde N, v} \leq 4 \delta \log N, \ \xi^*_{N,v} = 1  \right\} ,
 $$
where $\xi^*_{N,w} = 1$ if $\xi_{N,z} = 1$ for all $|z-w| \le 1$; and $\xi^*_{N,w} = 0$ otherwise. Note $\{ \xi^*_{N,w}, w \in V_N \}$ is $(q+2)$-dependent, and $P(\xi^*_{N,w} = 1) \ge p_0 (\d, q+2)$ if we set $p_1 (\d, q)$ such that $1 - 9 (1 - p_1 (\d, q)) \ge p_0 (\d, q+2)$, i.e.
 $$
p_1 (\d, q) \ge 1 - (1-p_0 (\d,q+2)) / 9 .
 $$
By Theorem~\ref{cardinality} (applied to the box $\tilde V_{\tilde N}$), since $r$ is large enough, with probability at least $1 - 6 e^{-k(r+1)/10} (\ge 1 - 6 e^{-\b_1 \ell})$, we have the event
 $$
E_3 : = \{ | \Psi(P)| \ge \tilde N^{1-\d} \mbox{ simultaneously for all paths } P \subseteq \tilde V_{\tilde N} \mbox{ with distance at least } \tilde N^{1-\kappa} \}
 $$
happens, for $K \ge K_1 (\d, q) : = K_0 (\d, q+2)$ and $p \ge p_1 (\d, q)$. Then, we have $\P (E_1 \cap E_2 \cap E_3) \ge 1 - 9 e^{- \zeta_1 \ell}$. In order to prove \eqref{crosswhp}, it remains to prove that $E_1 \cap E_2 \cap E_3 \subset E$.

Suppose $E_1, E_2, E_3$ all happen. We call a sequence of vertices $P_0, P_1, \ldots, P_h$ a {\em $*$-sequence} if $|P_i - P_{i+1}| = 1$ for all $0\leq i< h$, and a {\em $*$-path} if in addition $|P_i - P_{i^\prime}| > 1$ for $|i^\prime - i| > 1$.  Suppose $\C \in {\LR}_c (\acute{B})$. Then it  contains a $*$-sequence $\tilde P = (\tilde P_0, \tilde P_1, \ldots, \tilde P_{\tilde h})$ from the top boundary of $\acute{B}$ to the bottom one. We extract a $*$-path $P^* = (P^*_0, P^*_1, \ldots, P^*_h) = (\tilde P_{s_0}, \tilde P_{s_1}, \ldots, \tilde P_{s_h})$ by defining $s_0 := 0$, $s_{i+1} : = \inf \{ s \ge s_i + 1 : |\tilde P_t - P^*_i| > 1, \forall t > s\}$ till $P^*$ hits the bottom boundary. We then take a path $P$ such that $P \supset P^*$, by inserting one vertex between $P^*_s$ and $P^*_{s+1}$ when they are not neighbours step by step for $s=0,1,\cdots, h-1$. Note if $P^*_s$ and $P^*_{s+1}$  are not neighbours, they have two common neighbours, at most one of which is used in the steps $0,1,\ldots, s-1$. So we can succeed in finding the path $P$. Let
 $$
Q : = \{ v \in P^* : v \mbox{ is in } \Psi (P) \mbox{ or neighbouring } \Psi(P) \} .
 $$
Then we have $|Q| \geq |\Psi(P)|/4$. For any  $w \in Q$, by the definition of $\Psi (P)$ and $Q$, we have $\varphi_{\tilde N, w} \le 5 \d \log N$ on $E_1$. This together with $E_2$ implies that $\varphi_{N,w} \le 7 \d \log N$. By the definitions of $\Psi (P)$, $\xi^*_{N,w}$ and $Q$, we have $\xi_{N,w} = 1$, $\forall w \in Q$. It follows that
 $$
|\{ w \in \C : \varphi_{N, w} \le 7 \d \log N, \ \xi_{N,w} = 1 \}| \ge |Q| \ge \tilde N^{1-\d} / 4,
 $$
where the last inequality holds on $E_3$. The same reasoning implies that the above event also holds for all $\C \in {\cal UD}_c (B)$. Hence \eqref{crosswhp} holds.

By min-cut-max-flow theorem, on $E$, there are at least $\tilde N^{1-\d}/4$ disjoint good paths in $\UD (B)$. The shortest one, denoted by $Q^V$, has cardinality at most $4 L^2 / \tilde N^{1-\d} \le 4 L^{1+\d}$, since $|B| = L^2$. By the same reasoning, there exists a good path $Q^H$ in $\LR (\acute{B})$ with cardinality at most $16 \times L^{1 + \d}$. This completes the proof of the lemma.
\end{proof}

\begin{proof}[Proof of Proposition~\ref{goodpath}]
Recall $u = (u_1, u_2)$ and $v = (v_1, v_2)$. Without loss of generality, we suppose $v_1 - u_1 = |u-v| \ge N^{1-\k }$. Let
 $$
L_u = \{ z = (z_1, z_2) : z_1 = u_1 + \lceil N^{1 - 2 \k} \rceil \}, \ \ \ L_v = \{ z = (z_1, z_2) : z_1 = v_1 - \lceil N^{1 - 2 \k} \rceil \} .
 $$
We will first show that one can find a good path in $V_N$ from $u$ to $L_v$, using a Russo-Seymour-Welsh type of technique.  Without loss of generality, we suppose $u_2 < N/2$. Recall $\ell_0 = k  \lfloor \d^2 m \rfloor$. Let $B_{u,\ell}$ (for  $\ell \ge \ell_0$) be the box of side length $2^\ell$ and with the left bottom corner $u + ( 2^{\ell + 1} - 2^{\ell_0 + 1}, 0)$. Then, $B_{u, \ell+1}$ and $\acute{B}_{u,\ell}$ share the same right bottom corner.  Let $\ell_2$ be the minimum $\ell$ such that $\acute{B}_{u,\ell}$ intersects $L_v$, i.e.,
 $$
\ell_2 : = \lceil \log_2 \frac{v_1 - u_1 - \lceil N^{1-2 \k} \rceil + 2^{\ell_0+ 1} + 1 }{6} \rceil .
 $$

By a simple union bound, we see that
 $$
\P ( E_1^c ) \leq 2^{2\ell_0} 2 e^{- 24 \delta^2 (\log^2 2) n} \leq e^{-\delta^2 n} ,
 $$
where
 $$
E_1 = \left\{  \varphi_{N,z} \le 7 \d \log N, \ \forall  z \in B_{u, \ell_0} \right\} .
 $$
Denote
 $$
E_2 = \left\{ \mbox{there exists an open path in } B_{u, \ell_0} \mbox{ from } u \mbox{ to } u + (0, 2^{\ell_0}) \right\} .
 $$
Note $u+ (0, 2^{\ell_0})$ is the left top corner of $B_{u, \ell_0}$. Then on $E_2^c$, one can find a closed cut set (and thus a $*$-path) in $B_{u, \ell_0}$, separating $u$ and $u+(0, 2^{\ell_0})$. On the one hand, there are at most $C_h = 2 h \times 8^h$ possible $*$-paths in $B_{u, \ell_0}$ with length $h$. On the other hand, suppose $P^*$ is a $*$-path with length $h$. Then we can find $Q \subset P^*$ such that $|Q| \ge \lceil |P^*| / (2q+1)^2 \rceil$ and each pair of $Q$ has distance at least $q+1$. It follows that
 $$
\P (E_2^c) \le \sum_{h=1}^\infty C_h \times (1-p)^{\lceil \frac h {(2q+1)^2} \rceil} \le \sum_{h=1}^\infty 2h \times ( 8 (1-p)^{ \frac 1 {(2q+1)^2} } )^h \le \frac{ 16 (1-p)^{ \frac 1 {(2q+1)^2} } }{(1-  8 (1-p)^{ \frac 1 {(2q+1)^2} } )^2} \le 64 (1-p)^{\frac 1 {(2q+1)^2}},
 $$
provided that $8 (1-p)^{\frac 1 {(2q+1)^2}} \le \frac 1 2$. Thus the above inequality holds for $p$ great than or equal to
 $$
p_1 (\d, p) := \max \left\{ 1 - \frac 1 9 ( 1-p_0 (\d, q+2)) , \ 1 - 16^{- (2q+1)^2} \right\} .
 $$
Denote
 $$
E_3 = \{ \mbox{there exists a } B_{u, \ell}\mbox{-crossing}, \forall \ell = \ell_0, \ldots, \ell_2 \} .
 $$
Then by the above reasoning as well as Lemma \ref{onetimesfour}, we conclude
 $$
\P(E_1 \cap E_2 \cap E_3) \ge 1 - e^{-\d^2 n} - 64 (1-p)^{\frac 1 {(2q+1)^2}}  - \frac {9}{1 - e^{-\zeta_1}} e^{-\zeta_1 k \lfloor \d^2 m \rfloor} \ge 1 - 64 (1-p)^{\frac 1 {(2q+1)^2}} - \frac 1 3 e^{- \zeta n},
 $$
where the last inequality holds for $\zeta < \zeta_1 \d^2$.

Next, we are going to check that on the event $E = E_1 \cap E_2 \cap E_3$ we can find a ``short" good path from $u$ to $L_v$. Assume $E$ occurs. Denote  $Q^{u,\ell} = Q^{u,\ell,V} \cup Q^{u,\ell,H}$ the $B_{u,\ell}$-crossings. On $E_2$, we can find an open path in $B_{u, \ell_0}$ from $u$ to $u + (0,2^{\ell_0})$. It is furthermore a good path since $E_1$ occurs. We choose it as  $Q^{u, \ell_0,V}$.  Then,
 $$
\( \cup_{\ell=\ell_0}^{\ell_2 - 1} Q^{u,\ell} \) \cup Q^{u,\ell_2,V} \cup \( Q^{u, \ell_2, H} \cap \{ z : z_1 \le v_1 - \lceil N^{1-2\k} \rceil \} \)
 $$
contains a good path $P^u$ from $u$ to  $L_v$. In addition,
 $$
P^u \subseteq V_N \cap \{ z : z_1 \le v_1 - \lceil N^{1-2\k} \rceil \}
 $$
since $B_{u, \ell_2}$ is located at the left of $L_v$ and has height $2^{\ell_2} < \frac N 2 < N - u_1$. The good path $P^u$ has cardinality
 $$
|P^u| \le 20 \( (2^{\ell_0})^{1 + \d} + \ldots + (2^{\ell_2})^{1 + \d}  \) = 20 \mbox{$\sum_{\ell=\ell_0}^{\ell_2}$} (2^{1+\d} )^\ell \le 20 \times 2^{(1+\d) (\ell_2+1)} \le 80  N^{1+\d} .
 $$
By symmetry, with probability at least $1 - 64 (1-p)^{\frac 1 {(2q+1)^2}} - \frac 1 3 e^{- \zeta n}$, there exists a good path $P^v$ from $v$ to $L_u$, which has cardinality at most $ 80 N^{1+\d}$. Denote
 $$
B_{u,v} : = \{ w \in  V_N : u_1 + \lceil N^{1 - 2 \k } \rceil \le w_1 \le v_1 - \lceil N^{1 - 2 \k } \rceil  \} .
 $$
Then each path in $\LR (B_{u, v})$ has distance at least $2^{(1 - \k)  n} - 2 \times 2^{(1 - 2 \k ) n} > 2^{(1 - 2 \k ) n}$. By Theorem \ref{cardinality} and the same reasoning in the proof of Lemma~\ref{onetimesfour}, with probability at least $1 - \frac 1 3 e^{- \zeta n}$ there exists a good path $P^V$ in $\UD (B_{u,v})$ with cardinality at most $4 N^{1+\d}$. Provided paths $P^u$, $P^v$ and $P^V$,  there exists a good path $P$ connecting $u$ and $v$ such that
 $$
|P|  \le 160 \times N^{1+\d} + 4 N^{1+\d} \le N^{1 +2 \d} .
 $$
This event happens with probability at least $1 - 128 (1-p)^{\frac 1 {(2q+1)^2}} - e^{-\b n}$ for $\b < \b_1 \d^2$.

Finally, we decompose $P$ into three parts $P^C$, $P^F$ and $Q$ (which may overlap) as stated in the proposition. Let $P^C$ consist of $Q^{u, \ell_0}$ as well as its analog $Q^{v, \ell_0}$, $Q$ consist of $Q^{u,\ell}$ as well as the analogs $Q^{v,\ell}$ for all $\ell=\ell_0 + 1, \ldots, \ell_1$, and $P^F = P \setminus (P^C \cup Q)$. We would like to mention that $B_{u, \ell_1 + 1}$ is totally at the right of $L_u$, and $B_{v, \ell_1 + 1}$ is totally at the left of $L_v$.

\begin{figure}[h!]
 \includegraphics[width=17cm]{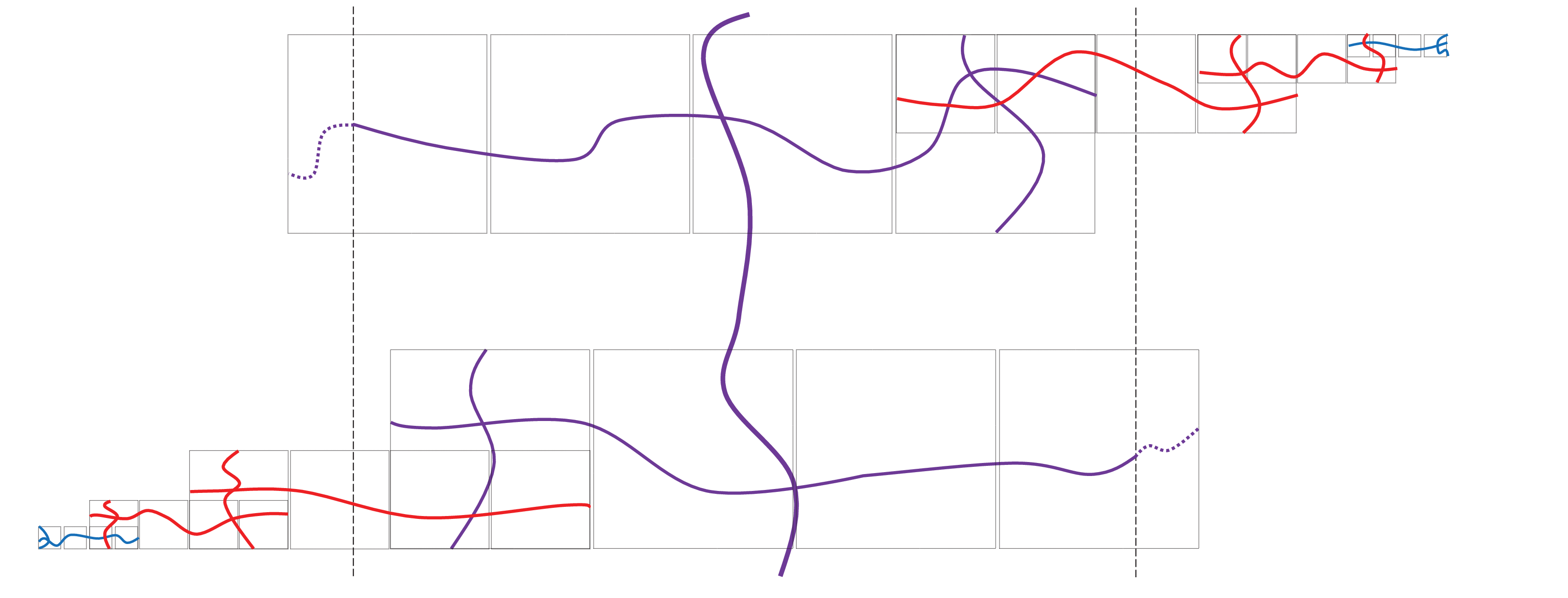}
\\ \vspace{-1cm}  \caption{$u$ is at the left bottom corner of the figure, and $v$ is at the right top corner. The left dashed line is $L_u$, and the right one is $L_v$.  $P^C$ is blue, $Q^\ell$'s are red, and $P^F$ is purple, where the parts with dot-line will be removed. }\label{G.RSW} \end{figure}

Observe that $|P^C| \le |Q^{u, \ell_0}| + |Q^{v, \ell_0}| \le 40 \times 2^{\ell_0 (1 + \d)} \le 40 \times  2^{(1+\d) \d^2 n} = 40N^{(1+\d) \d^2} $. This, together with the facts that $P^F \subset P$, and $Q^\ell : = Q^{u,\ell} \cup Q^{v,\ell}$ has cardinality at most $40 \times  2^{(1+\d) \ell}$, implies (i) of Proposition \ref{goodpath}. For (ii), suppose $z \in Q^{u,\ell}$, for $\ell = \ell_0 + 1, \ldots, \ell_1$. We have $|z-u| \ge 2^\ell$, and $|z-v| \ge \frac 1 2 N^{1 - \kappa}$. This together with Lemma \ref{covariance} implies that
 $$
\s_{z,u} + \s_{z,v} \le   n (1+\k) + 1 - \ell .
 $$
The same result holds for $z \in Q^{v,\ell}$ by symmetry. For any $z \in P^F$, note that $B_{u,\ell_1+1}$ is located at the right of $L_u$, and that $P^u$ is located at the left of $L_v$. It follows that $P^u \setminus (P^C \cup Q) \subset B_{u,v}$. By symmetry, $P^v \setminus (P^C \cup Q) \subset B_{u,v}$. These together with $P^V \in \UD (B_{u,v})$ imply that $P^F \subset B_{u,v}$.
It follows that for any $z \in P^F$, $|z-u|, |z-v| \ge N^{1 - 2 \k}$, which implies that $\s_{z,u} + \s_{z,v} \le 4 n \k$. Thus, we complete the proof of Proposition \ref{goodpath}.
\end{proof}

\subsubsection{Proof of Proposition \ref{givenendvalues}}

Set
 $$
a_z = \frac {km \s_{z,u} -  \s_{z,v} \s_{u,v}}{(km)^2 - \s_{u,v}^2} , \mbox{ and } b_z = \frac{km \s_{z,v} - \s_{z,u} \s_{u,v}}{(km)^2 - \s_{u,v}^2} \,.
 $$
We see that $G_z : = \varphi_{N,z} - (a_z \varphi_{N,u} + b_z \varphi_{N,v})$ is independent of $(\varphi_{N,u}, \varphi_{N,v})$. Denote
 $$
\phi_{N,z} = \varphi_{N,z} - ( a_z \varphi_{N,u} + b_z \varphi_{N,v})  +  (a_z x + b_z y) .
 $$
Then, the conditional distribution of $\{ \varphi_{N,z} : z \in V_N \} $ given $\varphi_{N,u} = x, \varphi_{N,v} = y$ is the same as the distribution of $\{ \phi_{N,z} : z \in V_N \}$. That is,
 $$
{\cal L} ( \{ \varphi_{N,z} : z \in V_N \} \mid \varphi_{N,u} = x, \varphi_{N,v} = y) = {\cal L} (\{ \phi_{N,z} : z \in V_N \}) .
 $$
Hence, we consider the FPP distance between $u$ and $v$ according to the Gaussian field $\{ \phi_{N,z} : z \in N \}$, which is denoted by $d_{\gamma, \phi} (u,v)$.

Let $P$ be the good path (with respect to the field $\{\varphi_{N,w}: w\in V_N\}$) given in Proposition \ref{goodpath}, which exists with probability at least $1- e^{- \b n}$ (recall $p=1$). Clearly, $d_{\gamma, \phi} (u,v) \le \sum_{z \in P} e^{\gamma \phi_{N,z}}$. For all $z \in P$, by the definition of the good path and the assumption $x,y \le 2 \gamma \log N$,
  $$
\phi_{N,z} \le (1 + |a_z| + |b_z|) 7 \d \log N + (|a_z| + |b_z|) 2 \gamma \log N .
 $$
By Lemma \ref{covariance} and $|u-v| > N^{1-\k}$, one has $km - \s_{u,v} \ge n ( 1-\k )$. Therefore,
 \begin{equation} \label{Eq.azplusbz}
|a_z| + |b_z| \le \frac {(km +\s_{u,v})(\s_{z,u} + \s_{z,v})}{(km)^2 - \s_{u,v}^2}  =  \frac { \s_{z,u} + \s_{z,v} }{km - \s_{u,v}} \le \frac {\s_{z,u} + \s_{z,v}}{n (1-\k)} ,
 \end{equation}
which is less than $\frac 2 {1-\k}$ since $\s_{z,u}, \s_{z,v} \le n$. It follows that
 $$
e^{\gamma \phi_{N,z}} \le N^{ \frac {3-\k} {1-\k} 7 \d \gamma} N^{(|a_z| + |b_z|) 2 \gamma^2} .
 $$

Take $\d$ (consequently, $\k$) small enough, and we only need to check
 \begin{equation} \label{dphi}
\sum_{z \in P} N^{(|a_z| + |b_z|) 2 \gamma^2} < N^{1 + \frac1 2 \e}.
 \end{equation}
To this end, let $P^C$, $P^F$, $Q^\ell$, (for $\ell = \ell_0 + 1, \ldots, \ell_1$) be as given in Proposition \ref{goodpath}. Then
 $$
\sum_{z \in P} N^{(|a_z| + |b_z|) 2 \gamma^2} \le \sum_{z \in P^C} N^{(|a_z| + |b_z|) 2 \gamma^2} + \sum_{z \in P^F} N^{(|a_z| + |b_z|) 2 \gamma^2} + \sum_{\ell=\ell_0 + 1}^{\ell_1} \sum_{z \in Q^\ell} N^{(|a_z| + |b_z|) 2 \gamma^2}  .
 $$
Next we will show that each term in the right hand side above is less than $\tfrac 1 3 N^{1 + \frac 1 2 \e}$, which implies \eqref{dphi}.

For $z \in P^C$, we use $|a_z| + |b_z| \le \frac 2 {1-\k}$. This together with  (i) of Proposition \ref{goodpath}  implies that
 $$
\sum_{z \in P^C}  N^{(|a_z| + |b_z|) 2 \gamma^2} \le 40 N^{(1+\d)\d^2 + \frac 2 {1-\k} \times 2 \gamma^2 } \le \frac 1 3 N^{1 + \frac 1 2 \e},
 $$
where we use $\gamma < \frac 1 2$ as well as $\d$ and $\k$ being small enough.

For $z \in P^F$,  Proposition \ref{goodpath}  and \eqref{Eq.azplusbz} imply that
 $$
\sum_{z \in P^F}  N^{(|a_z| + |b_z|) 2 \gamma^2} \le N^{1 + 2 \d + \frac {4 \k} {1 - \k} \times 2 \gamma^2} \le \frac 1 3 N^{1 + \frac 1 2  \e} .
 $$

For $z \in Q^\ell$, Proposition \ref{goodpath} and \eqref{Eq.azplusbz} imply that
 $$
\sum_{z \in Q^\ell}  N^{(|a_z| + |b_z|) 2 \gamma^2} \le 40 \times 2^{(1+\d) \ell} N^{  \( \frac {n(1+\k) + 1 - \ell} {n (1-\k)} \) 2 \gamma^2 } = 40 N^{\frac {1+\k}{1-\k} 2 \gamma^2} 2^{\frac 1{1-\k} 2 \gamma^2} \( 2^{1+\d - \frac {1}{1-\k} 2 \gamma^2} \) ^\ell.
 $$
It follows that for $C = \tfrac {40 \times  2^{\frac 1{1-\k} 2 \gamma^2}} { 2^{1+\d - \frac {1}{1-\k} 2 \gamma^2} -1} $,
 \begin{eqnarray*}
\sum_{\ell= \ell_0+1}^{\ell_1} \sum_{z \in Q^\ell} N^{(|a_z| + |b_z|) 2 \gamma^2}
 & \le &
C N^{ \frac {1+\k}{1-\k} 2 \gamma^2 }   \(  2^{1+\d - \frac {1}{1-\k} 2 \gamma^2} \) ^{\ell_1+1} .
 \end{eqnarray*}
Recall $\ell_1 = \lfloor \log_2 ( \lceil N^{1-2\k} \rceil + 2^{\ell_0 + 1}) \rfloor - 1$, where $\ell_0 \le \d^2 n$. Thus, $2^{\ell_1 + 1} \le \lceil N^{1 - 2 \k} \rceil + 2 N^{\d^2} \le N$. Therefore,
 $$
\sum_{\ell=\ell_0+1}^{\ell_1} \sum_{z \in Q^\ell} N^{(|a_z| + |b_z|) 2 \gamma^2} \le C N^{ \frac {1+\k}{1-\k} 2 \gamma^2 } N^{1+\d - \frac {1}{1-\k} 2 \gamma^2  } = C  N^{1+\d +  \frac {\k}{1-\k} 2 \gamma^2 }  \leq  \frac 1 3 N^{1 + \frac 1 2 \e},
 $$
provided that $N$ is sufficiently large. Altogether, this completes the proof of the proposition.

\end{document}